\documentclass[11pt,a4paper]{amsart}
\usepackage[latin1]{inputenc}
\usepackage[english]{babel}
\usepackage{amsmath,amssymb,amsthm}
\usepackage[dvipsnames]{xcolor}

\usepackage{enumerate}
\usepackage{cite}
\usepackage{mathtools}

\theoremstyle{plain}
\newtheorem{theorem}{Theorem}[section]
\newtheorem{lemma}[theorem]{Lemma}
\newtheorem{corollary}[theorem]{Corollary}

\newtheorem{proposition}[theorem]{Proposition}

\theoremstyle{definition}
\newtheorem{definition}[theorem]{Definition}

\newtheorem{remark}[theorem]{Remark}

\let\phi\varphi
\let\epsilon\varepsilon
\newcommand{\norm}[1]{\Vert#1\Vert}
\newcommand{\bignorm}[1]{\bigl\Vert#1\bigr\Vert}
\newcommand{\Bignorm}[1]{\Bigl\Vert#1\Bigr\Vert}

\newcommand{\decnorm}[1]{\Vert#1\Vert_{dec}}

 \newcommand{\A}{\mbox{${\mathcal A}$}}
   \newcommand{\B}{\mbox{${\mathcal B}$}}
   
   \newcommand{\D}{\mbox{${\mathcal D}$}}

   \renewcommand{\H}{\mbox{${\mathcal H}$}}

    \newcommand{\M}{\mbox{${\mathcal M}$}}
   \newcommand{\N}{\mbox{${\mathcal N}$}}

\newcommand{\Ndb}{\ensuremath{\mathbb{N}}}
\newcommand{\Cdb}{\ensuremath{\mathbb{C}}}

\begin{document}
\title{On factorization of separating maps on noncommutative $L^p$-spaces}

\author[C. Le Merdy]{Christian Le Merdy$^1$}

\address{Laboratoire de Math\'ematiques de Besan\c{c}on,
Universite Bourgogne Franche-Comt\'e, France}
\email{\texttt{christian.lemerdy@univ-fcomte.fr}}

\author[S. Zadeh]{Safoura Zadeh}
\address{Westf\"{a}lische Wilhelms-Universit\"{a}t M\"{u}nster, 
M\"{u}nster, Germany\,\&\newline
\indent 
Max-Planck-Institut f\"{u}r Mathematik,
Vivatsgasse 7, 53111, Bonn, Germany} 
\email{\texttt{jsafoora@gmail.com}}


\subjclass[2000]{}
  
\begin{abstract} 
For any semifinite von Neumann algebra
$\M$ and any $1\leq p<\infty$, we introduce 
a natutal $S^1$-valued noncommutative
$L^p$-space $L^p(\M;S^1)$.
We say that a bounded  map
$T\colon L^p(\M)\to L^p(\N)$ is 
$S^1$-bounded (resp. $S^1$-contractive)
if $T\otimes I_{S^1}$ extends to a 
bounded (resp. contractive) map $T\overline{\otimes} I_{S^1}$
from $ L^p(\M;S^1)$ into $L^p(\N;S^1)$.
We show that any completely positive map
is $S^1$-bounded, with 
$\norm{T\overline{\otimes} I_{S^1}}=\norm{T}$.
We use the above as a tool to investigate the separating maps
$T\colon L^p(\M)\to L^p(\N)$ which admit a direct Yeadon
type factorization, that is, maps for which
there exist a $w^*$-continuous $*$-homomorphism 
$J\colon\M\to\N$,
a partial isometry $w\in\N$ and a 
positive operator $B$ affiliated
with $\N$ such that $w^*w=J(1)=s(B)$,
$B$ commutes with the range of $J$, and 
$T(x)=wBJ(x)$ for any $x\in \M\cap L^p(\M)$.
Given a separating isometry $T\colon L^p(\M)\to L^p(\N)$,
we show that $T$ is $S^1$-contractive
if and only if it  admits a direct Yeadon
type factorization. We 
further show that if
$p\not=2$, the above holds true if and only if 
$T$ is completely contractive.
\end{abstract}

\date{\today}

\maketitle

\section{Introduction}\label{sec1}
Let $\M,\N$ be two semifinite von Neumann algebras. For 
any $1\leq p<\infty$,
let $L^p(\M)$ and $L^p(\N)$ denote their
associated noncommutative $L^p$-spaces. A bounded map 
$T\colon L^p(\M)\to  L^p(\N)$
is called separating if for any 
$x,y\in L^p(\M)$ such that $x^*y=xy^*=0$, we have
$T(x)^*T(y)=T(x)T(y)^*=0$. Separating maps are a noncommutative 
analog of Lamperti operators, that is, operators 
on classical (=commutative) $L^p$-spaces preserving disjoint supports.
We refer  to \cite{Lam, Kan, Are, Pel} 
for information and deep results on Lamperti operators.

In the noncommutative setting, pairs $(x,y)$
such that $x^*y=xy^*=0$ were first considered on Schatten classes $S^p$
in \cite{Ara}, as a tool to describe onto surjective isometries on 
$S^p$ for $1\leq p\not=2 <\infty$. Later on, separating maps were used
either implicitly or explicitly, and with different names, 
in \cite{AF1,AF2} (see also \cite{Li}) and in Yeadon's paper \cite{Y}
providing a full description of isometries $L^p(\M)\to L^p(\N)$,
for $1\leq p\not=2 <\infty$.

Recently the two authors \cite{LMZ} and, independently, 
G. Hong, S. K. Ray and S. Wang \cite{HRW} established 
the following characterization property. A bounded map
$T\colon L^p(\M)\to L^p(\N)$ is separating if and only if 
there exist a $w^*$-continuous Jordan homomorphism $J\colon\M\to\N$,
a partial isometry $w\in\N$ and a positive operator $B$ affiliated
with $\N$ such that $w^*w=J(1)=s(B)$, the support of $B$,
$B$ commutes with the range of $J$, and 
\begin{equation}\label{Yeadon}
T(x)=wBJ(x),\qquad x\in \M\cap L^p(\M).
\end{equation}
This remarkable factorization property was discovered
by Yeadon in the above mentioned paper. Indeed he showed in \cite{Y}
that for $p\not=2$, any linear isometry
$T\colon L^p(\M)\to L^p(\N)$ is separating 
and further admits a factorization of the type (\ref{Yeadon}).
In reference to this seminal work, we call
(\ref{Yeadon}) a Yeadon type factorization of $T$.
It turns out that if $T$ is separating, the triple
$(w,B,J)$ in its Yeadon type factorization is unique.

We note that analogs of separating maps are currently investigated 
in other settings. On the one hand, they are used on 
general noncommutative functions spaces, in order to obtain a Yeadon type
description of isometries on a large class of such spaces \cite{HSZ}. 
On the other hand, they are investigated in operator algebras 
(the case $p=\infty$)
and play a fundamental role in the classification
of nuclear $C^*$-algebras, see \cite{W} and the references therein.

The present paper looks at separating maps $T\colon L^p(\M)\to L^p(\N)$ 
for which the Jordan homomorphism $J$ in the 
Yeadon type factorization is actually a $*$-homomorphism
(equivalently, is multiplicative).
We say that $T$ has a direct Yeadon type factorization
in this case.
The first motivation for considering this notion
is a result by M. Junge, D. Sherman and Z.-J. Ruan 
\cite[Proposition 3.2]{JRS}
which asserts that for $p\not=2$, a linear isometry
$L^p(\M)\to L^p(\N)$ is a complete isometry
if and only if it has a direct Yeadon type factorization. 
The second motivation is the $L^2$-case. In \cite[Theorem 4.2]{LMZ},
we proved that an isometry $T\colon L^2(\M)\to L^2(\N)$
is separating (equivalently, has a Yeadon type factorization)
if and only if $T\otimes I_{\ell^1}$ extends to a 
contractive map $L^2(\M;\ell^1)\to L^2(\N;\ell^1)$. Here
$L^2(\M;\ell^1)$ and $L^2(\N;\ell^1)$ denote Junge's $\ell^1$-valued
non commutative $L^2$-spaces from \cite{J}. 

We introduce $S^1$-valued noncommutative $L^p$-spaces $L^p(\M;S^1)$,
which naturally extend previous constructions from \cite{J,P5}. 
We say that a bounded map $T\colon L^p(\M)\to L^p(\N)$ is 
$S^1$-bounded (resp. $S^1$-contractive)
if $T\otimes I_{S^1}$ extends to a 
bounded (resp. contractive) map
$$
T\overline{\otimes} I_{S^1}
\colon L^p(\M;S^1)\longrightarrow L^p(\N;S^1).
$$
When $\M,\N$ are hyperfinite, $S^1$-boundedness coincides
with complete regularity in the sense of \cite{P2}
(see also \cite{AK,JR}) and $\norm{T\overline{\otimes} I_{S^1}}=
\norm{T}_{reg}$.
We prove that any map with
a direct Yeadon type factorization is $S^1$-bounded, with
$\norm{T\overline{\otimes} I_{S^1}}=\norm{T}$ (see Proposition
\ref{4.2}). Our main result 
is that conversely, any $S^1$-contractive separating isometry 
admits a direct Yeadon type factorization (see Theorem \ref{jan5}). 
The resulting 
statement (see Corollary \ref{JRS-L2}) that an isometry $T\colon L^2(\M)\to L^2(\N)$ 
is $S^1$-contractive if and only if it admits 
a direct Yeadon type factorization is both an $L^2$-version
of \cite[Proposition 3.2]{JRS} and a matricial version
of \cite[Theorem 4.2]{LMZ}. 

The spaces $L^p(\M;S^1)$
and $S^1$-boundedness are investigated in Section 3. We prove in passing that
any completely positive map $T\colon L^p(\M)\to L^p(\N)$
is $S^1$-bounded, with $\norm{T\overline{\otimes} I_{S^1}}=\norm{T}$
(see Theorem \ref{ls1}).

We also establish comparisons between 
direct Yeadon type factorizations and complete boundedness.
After observing that any separating map
$T\colon L^p(\M)\to L^p(\N)$ with a direct Yeadon type factorization
is completely bounded, with $\norm{T}_{cb}=\norm{T}$ (see Proposition
\ref{4.4}), we show
that conversely if $p\not=2$, any completely contractive
isometry $T\colon L^p(\M)\to L^p(\N)$ admits a direct 
Yeadon type factorization (see Theorem \ref{5.6}). This result
strengthens \cite[Proposition 3.2]{JRS}.

\section{Noncommutative $L^p$-spaces and representations of 
matrix spaces}


In this section, we give some background and preliminary
facts on  noncommutative $L^p$-spaces built over semifinite 
von Neumann algebras.

Let $\M$ be a semifinite von Neumann algebra equipped 
with a normal semifinite faithful (n.s.f.) trace \cite[Definition V.2.1]{Ta}. 
Except otherwise stated, this trace will be 
denoted by $\tau_{_{\tiny{\M}}}$. Assume that $\M\subset B(\H)$ acts 
on some Hilbert space $\H$. Let $L^0(\M)$ 
denote the $*$-algebra of all closed densely defined (possibly unbounded)
operators on $\H$, which are
$\tau_{_{\tiny{\M}}}$-measurable. 
Then for any $0< p<\infty$, the noncommutative $L^p$-space $L^p(\mathcal{M})$, 
associated with $(\mathcal{M},\tau_{\tiny{\M}})$, can be defined as
$$
L^p(\mathcal{M}):=\bigl\{x\in L^0(\mathcal{M})
\,:\,\tau_{\tiny\M}(\lvert x\rvert^p)<\infty\bigr\}.
$$
We set 
$\|x\|_p:=\tau_{\tiny\M}(\left\lvert x\rvert^p\right)^{\frac{1}{p}}$
for any 
$x\in L^p(\M)$. If $p\geq 1$, $L^p(\M)$ equipped with $\norm{\,\cdotp}_p$
is a Banach space.
The reader is referred to \cite{JX,PX,T}
and the references therein for details on the algebraic operations
on $L^0(\mathcal{M})$ and the construction of $L^p(\M)$, and  for
further properties.

We let $L^\infty(\M)=\M$ for convenience and for any $x\in \M$, we let 
$\norm{x}_\infty$ denote its operator norm. 
We recall that if $0< p,q,r \leq\infty$ are such that $\frac{1}{r}
=\frac{1}{p} + \frac{1}{q}$, then for any $x\in L^p(\M)$
and $y\in L^q(\M)$, the product $xy$ belongs to $L^r(\M)$, with
$\norm{xy}_r\leq  \|x\|_p\|y\|_q$. 
In particular, 
for any $1\leq p<\infty$, let $p'=\frac{p}{p-1}$ 
be the conjugate number of $p$. Then 
$xy$ belongs to $L^1(\M)$
for any $x\in L^p(\M)$ and $y\in L^{p'}(\M)$.
Further the duality pairing  
$$
\langle x,y\rangle=\tau_{\tiny{\M}}(xy),\qquad  x\in L^p(\mathcal{M}),\ y\in L^{p^{\prime}}(\mathcal{M}),
$$
yields an isometric isomorphism $L^{p}(\M)^*= 
L^{p'}(\M)$.
In particular, we may identify $L^1(\M)$ with the (unique) 
predual of $\mathcal{M}$. 
These duality results will be used without further reference in the paper.

For any $0<p\leq \infty$, we let $L^p(\M)^+$ denote the cone
of positive elements of $L^p(\M)$.

If $\A$ is a $w^*$-closed $*$-subalgebra of $\M$ 
such that the restriction of $\tau_{\tiny{\M}}$ to $\A^+$
is semifinite, then for any $0<p<\infty$,
we may define $L^p(\A)$ using this restriction
and $L^p(\A)$ isometrically embeds in $L^p(\M)$. 
In particular, for any projection $e$ in $\M$, 
the restriction of  $\tau_{\tiny{\M}}$ to the corner algebra
$e\M e$ is semifinite, and therefore we have a natural embedding
$$
L^p(e\M e)\,\subset\, L^p(\M).
$$

For any two von Neumann algebras $\M_1,\M_2$, 
we let $\M_1\overline{\otimes} \M_2$ 
denote their von Neumann tensor product \cite[Section
IV.5]{Ta}. If
$\tau_{\tiny{\M_1}}$ and $\tau_{\tiny{\M_2}}$ are n.s.f. traces on
$\M_1$ and $\M_2$, respectively, then $\tau_{\tiny{\M_1}}
\otimes \tau_{\tiny{\M_2}}$ uniquely extends to a n.s.f. trace
on $\M_1\overline{\otimes} \M_2$. Then for any 
any $0<p<\infty$, we have a natural embedding 
$L^p(\M_1)\otimes L^p(\M_2)\subset L^p(\M_1\overline{\otimes}\M_2)$, and
\begin{equation}\label{1-tensor}
\norm{x\otimes y}_p = \norm{x}_p\norm{y}_p,\qquad x\in L^p(\M_1),\, 
y\in L^p(\M_2).
\end{equation}
We further recall that $x\otimes y\in L^p(\M_1\overline{\otimes}\M_2)^+$
if $x\in L^p(\M_1)^+$ and 
$y\in L^p(\M_2)^+$.

We also note that the direct sum 
$\M_1\mathop{\oplus}\limits^{\infty}\M_2$ satisfies
\begin{equation}\label{DirectSum}
L^p\bigl(
\M_1\mathop{\oplus}\limits^{\infty}\M_2\bigr)
=L^p(\M_1)\mathop{\oplus}\limits^p L^p(\M_2)
\end{equation}
for any $0<p<\infty$.

We now fix some notations regarding matrix spaces.
Let $\H$ be a Hilbert space and let ${\rm tr}$ be the
usual trace on $B(\H)$. For any $0<p<\infty$, we
let $S^p(\H)$ denote the Schatten $p$-class of operators
on $\H$; this is the noncommutative $L^p$-space associated
with $(B(\H), {\rm tr})$. If $\H=\ell^2$, we simply
denote these spaces by $S^p$. For any $n\geq 1$, we let 
${\rm tr}_n$ denote the usual trace on $M_n$ and we let
$S^p_n$ denote the Schatten $p$-class of $n\times n$
matrices. We let $E_{ij}$, $1\leq i,j\leq n$, denote
the usual matrix units on $M_n$ and we let 
$I_n\in M_n$ be the identity matrix. Finally
whenever $\M$ is a semifinite von Neumann algebra
equipped with a n.s.f. trace $\tau_{\tiny{\M}}$, we let
$\tau_{\tiny\M,n} = {\rm tr}_n\otimes \tau_{\tiny{\M}}$
denote the natural trace on $M_n\overline{\otimes}\M$.
We note that $L^p(M_n\overline{\otimes}\M)$ can be naturally
regarded as a space of $n\times n$ matrices with values
in $L^p(\M)$. This brings us to the algebraic
identification  
\begin{equation}\label{Spn}
L^p(M_n\overline{\otimes}\M)
\,\simeq\, S^p_n\otimes L^p(\M).
\end{equation}

Let $T\colon L^p(\M)\to L^p(\N)$ be a bounded
operator between two noncommutative $L^p$-spaces. Following
usual terminology we say, using (\ref{Spn}),
that $T$ is completely bounded  if there exists a constant $C\geq 0$
such that 
$$
\bignorm{I_{S^p_n} \otimes T\colon L^p(M_n
\overline{\otimes}\M)\to L^p(M_n\overline{\otimes}\N)}\,\leq\,C
$$
for any $n\geq 1$. In this case we let $\norm{T}_{cb}$ denote the
smallest $C\geq 0$ satisfying this uniform estimate; it is called
the completely bounded norm of $T$. We say that $T$ is completely
contractive if $\norm{T}_{cb}\leq 1$. 
Further we say that $T$ is positive if it maps  $L^p(\M)^+$
into $L^p(\N)^+$ and we say that $T$ is 
completely positive maps if $I_{S^p_n} \otimes T$
is positive for any $n\geq 1$.
We recall that
in the case $p=2$, we have that any bounded $T\colon L^2(\M)\to L^2(\N)$
is automatically completely bounded, with $\norm{T}_{cb}=\norm{T}$.
This follows from the fact that $L^2(M_n
\overline{\otimes}\M)$ (resp. $L^2(M_n
\overline{\otimes}\N)$) coincides with the Hilbertian tensor product
of $S^2_n$ and $L^2(\M)$ (resp. $L^2(\N)$).

A positive map $T\colon (\M,\tau_{\tiny{\M}})\to (\N,\tau_{\tiny{\N}})$
is called trace preserving if $\tau_{\tiny{\N}}\circ T=\tau_{\tiny{\M}}$
on $\M^+$.

\begin{lemma}\label{TP-maps} 
Let $T\colon (\M,\tau_{\tiny{\M}})\to (\N,\tau_{\tiny{\N}})$ be a trace preserving 
$*$-homomorphism. Then for any $1\leq p <\infty$, the restriction of $T$ to $\M\cap L^1(\M)$
extends to a complete isometry $L^p(\M)\to L^p(\N)$.
\end{lemma}

\begin{proof}
Since $T$ is a $*$-homomorphism, $\vert(I_{M_n}\otimes T)(x)\vert^p
=(I_{M_n}\otimes T)(\vert x\vert^p)$ for any $x\in M_n\overline{\otimes}\M$. 
The result follows at once.
\end{proof}
 
We now give two elementary results on
the representation of matrix spaces into semifinite 
von Neumann algebras.

\begin{lemma}\label{last}
Suppose that $\M$ is a semifinite von Neumann algebra, let
$n\geq 1$ and let $\theta\colon M_n\to\M$ be a unital $*$-homomorphism. 
Then there exist a projection $e\in \M$ and a bijective 
$*$-homomorphism $\rho:\M\to M_n\overline{\otimes}(e\M e)$ such that 
$$
\left(\rho\circ\theta\right)(a)=a\otimes e, \qquad a\in M_n,
$$
and $\rho$ is trace preserving.
\end{lemma}

\begin{proof}
Let $e=\theta(E_{11})$, this is a projection. Since  $\theta$
is a unital $*$-homomorphism, the
family $\{\theta(E_{ij})\, :\, 1\leq i,j\leq n\}$ is a system of matrix units 
on $\M$. Hence as is well-known 
(see e.g. the proof of \cite[Proposition IV.1.8]{Ta}),
$x_{ij} := \theta(E_{1i})x\theta(E_{j1})$ belongs to $e\M e$ for any $x\in \M$ 
and any $1\leq i,j\leq n$, and the mapping
$$
\rho\colon \M\to M_n\overline{\otimes}(e\M e),\qquad \rho(x)=\sum_{i,j=1}^n E_{ij}\otimes x_{ij},
$$
is a bijective $*$-homomorphism. It is clear 
that $(\rho\circ\theta)(a)=a\otimes e$ for every $a$ in $M_n$. 

To check that $\rho$ is trace preserving, 
let $u_i=\theta(E_{i1})$ and $e_i=\theta(E_{ii})$ for all $1\leq i\leq n$. Then 
$u_iu_i^*=e_i$ and $e_1+\cdots + e_n =1$. Hence for any
$x\in \M^+$, $x_{ii}$ belongs to $(e\M e)^+$ for any $1\leq i\leq n$ and we have
$$
\sum_{i=1}^n\tau_{_{\tiny\M}}(x_{ii})=\sum_{i=1}^n\tau_{_{\tiny\M}}(u_i^*x
u_i)= \sum_{i=1}^n \tau_{_{\tiny\M}}(e_ix)=\tau_{_{\tiny\M}}(x).
$$ 
Therefore, $(tr_n\otimes\tau_{e\tiny{\M}e})
\circ\rho=\tau_{_{\tiny\M}}$ on $\M^+$.
\end{proof}

It is a classical fact that 
any non abelian von Neumann algebra contains
a copy of $M_2$. Here is a more precise statement
in the semifinite case.

\begin{lemma}\label{Contain-M2}
Let $\M$ be a non abelian semifinite 
von Neumann algebra. There exists a non zero
$*$-homomorphism $\gamma\colon M_2\to \M$
valued in $\M\cap L^1(\M)$.
\end{lemma}

In the above statement, the condition that $\gamma$ is 
valued in $\M\cap L^1(\M)$ does not come for free. 
Consider for example an infinite dimensional Hilbert space $H$ and 
let $\M=B(H\mathop{\oplus}\limits^2
H)\simeq M_2\overline{\otimes} B(H)$.
Then the mapping $a\mapsto a\otimes I_H$
is a $*$-homomorphism from $M_2$ into 
$\M$ and for any $a\in {M_2}^+$, $a\not=0$,
the trace of 
$a\otimes I_{H}$ is infinite.

\begin{proof}[Proof of Lemma \ref{Contain-M2}]
Let $\M=\M_1\mathop{\oplus}\limits^\infty \M_2$ be the
direct sum decomposition of $\M$ 
into a type I summand $\M_1$ and a type 
II summand $\M_2$ (see e.g. \cite[Section V]{Ta}). 

Assume that $\M_2\not=\{0\}$. According
to \cite[Lemma 6.5.6]{KR2}, there exist two equivalent mutually
orthogonal projections $e,f$ in $\M_2$ such that $e+f=1$. Then 
by \cite[Proposition V.1.22]{Ta} and its proof, 
$\M_2\simeq M_2\overline{\otimes} (e\M_2 e)$. Let $\epsilon\in e\M_2 e$ be a non
zero projection with finite trace. Then  
$\tau_{\tiny\M_2}(a\otimes \epsilon)={\rm tr_2}(a)\tau_{e{\tiny\M_2}e}(\epsilon)
<\infty$ for any $a\in {M_{2}}^+$. 
Hence the mapping 
$\gamma \colon M_{2}\to \M_2\subset \M$
defined by $\gamma(a)=  a\otimes \epsilon$
is a non zero $*$-homomorphism taking values in $L^1(\M)$.

If $\M_2 =\{0\}$, then $\M=\M_1$ is type I. Since 
$\M$ is non abelian, it follows from \cite[Theorem V.1.27]{Ta}
that there exists a Hilbert space $H$ with 
${\rm dim}(H)\geq 2$ and an abelian von Neumann
algebra $W$ such that $\M$ contains $B(H)\overline{\otimes} W$
as a summand. Let $e\in B(H)$ be a rank one projection
and define $\tau_{\tiny W}\colon W^+\to[0,\infty]$ by 
$\tau_{\tiny W}(z) = \tau_{\tiny\M}(e\otimes z)$. Then 
$\tau_{\tiny W}$ is a n.s.f. trace and $\tau_{\tiny\M}$
coincides with ${\rm tr}\otimes \tau_{\tiny W}$ on $B(H)^+\otimes
W^+$.
Let $\epsilon\in W$ be a non zero projection with finite trace. 
Then  it follows from above that
$\tau_{\tiny\M}(a\otimes \epsilon)<\infty$ for any finite rank $a\in B(H)^+$. 
Now let $(e_1,e_2)$ be
an orthonormal family in $H$. Then 
the mapping $\gamma\colon M_{2}\to \M_2$ taking
any $[a_{ij}]_{1\leq i,j\leq 2}$ to 
$\sum_{i,j} a_{ij} \,\overline{e_j}\otimes e_i\otimes \epsilon$
is a non zero $*$-homomorphism and
the restriction of $\tau_{\tiny \M}$ to the positive part of its 
range is finite. Hence $\gamma$ is valued in $L^1(\M)$.
\end{proof}

\section{$S^1$-boundedness}
In this section we introduce $S^1$-valued noncommutative 
$L^p$-spaces, in a way which extends the 
definition provided by \cite[Chapter 3]{P5} in the hyperfinite case. 
Then we introduce the notions 
of $S^1$-boundedness and $S^1$-contractivity for bounded maps between 
noncommutative $L^p$-spaces,
and we discuss the connection between $S^1$-boundedness and 
complete positivity.

We fix a semifinite von Neumann algebra $\M$.  We recall the definitions and basic properties
of column/row valued $L^p(\M)$-spaces for which we refer to \cite{PX0} (see also \cite{PX, J, LMZ}).
Let $\Lambda$ be an index set, and consider the Hilbert space $\displaystyle{\ell^2_{\Lambda}}$. 
For any $1\leq p\leq \infty$, let $L^p(\M;\{\ell^2_{\Lambda}\}_c)$ 
denote the space of all 
families $(b_{\lambda})_{\lambda\in\Lambda}$ of elements in $L^p(\M)$ such that the sums 
$\sum_{\lambda\in F}b_{\lambda}^*b_{\lambda}$, for finite $F\subset \Lambda$, are uniformly bounded
in $L^{\frac{p}{2}}(\M)$. Then for any such family, set 
$$
\bignorm{(b_\lambda)_\lambda}_{L^p(\tiny{\M};\{\ell^2_{\Lambda}\}_c)}\, 
=\,\sup\biggl\{\Bignorm{\sum_{\lambda\in F}b_{\lambda}^*b_{\lambda}
}_{\frac{p}{2}}^{\frac12}\biggr\},
$$
where the supremum runs over all finite $F\subset \Lambda$. This defines a norm
on $L^p(\M;\{\ell^2_{\Lambda}\}_c)$ and $L^p(\M;\{\ell^2_{\Lambda}\}_c)$ is complete. 

Likewise for any $1\leq p\leq \infty$, we let $L^p(\M;\{\ell^2_{\Lambda}\}_r)$ denote the space of all 
families $(a_{\lambda})_{\lambda\in\Lambda}$ of elements in $L^p(\M)$ such that the sums 
$\sum_{\lambda\in F}a_{\lambda}a_{\lambda}^*$, for finite $F\subset \Lambda$, are uniformly bounded
in $L^{\frac{p}{2}}(\M)$. This is a Banach space for the norm
$$
\bignorm{(a_\lambda)_\lambda}_{L^p(\tiny{M};\{\ell^2_{\Lambda}\}_r)}\, 
=\,\sup\biggl\{\Bignorm{\sum_{\lambda\in F}a_{\lambda}a_{\lambda}^*
}_{\frac{p}{2}}^{\frac12}\biggr\},
$$
where the supremum runs over all finite $F\subset \Lambda$. 
It is plain that $(a_{\lambda})_{\lambda\in\Lambda}$
belongs to $L^p(\M;\{\ell^2_{\Lambda}\}_r)$ if and only
if $(a_{\lambda}^*)_{\lambda\in\Lambda}$
belongs to $L^p(\M;\{\ell^2_{\Lambda}\}_c)$.

Let $(E_{\lambda,\mu})_{\lambda,\mu\in\Lambda}$ be the matrix units 
in $B(\ell^2_\Lambda)$ corresponding to the standard basis of
$\ell^2_\Lambda$. We may regard
any $z\in L^p(B(\ell^2_\Lambda)\overline{\otimes} \M)$ as a matrix
$(z_{\lambda,\mu})_{\lambda,\mu\in\Lambda}$ of elements in $L^p(\M)$,
with $E_{\lambda,\mu}\otimes z_{\lambda,\mu} = (E_{\lambda,\lambda}\otimes 1)z(E_{\mu,\mu}\otimes 1)$.
Then $L^p(\M;\{\ell^2_{\Lambda}\}_c)$ can be identified with any column subspace of 
$L^p(B(\ell^2_\Lambda)\overline{\otimes} \M)$. More precisely
fix any $\mu_0\in \Lambda$. If 
$z\in L^p(B(\ell^2_\Lambda)\overline{\otimes} \M)$ is such that 
$z_{\lambda,\mu}=0$ for any $\mu\not=\mu_0$ and any $\lambda$,
then $(z_{\lambda,\mu_0})_{\lambda\in\Lambda}$ belongs to 
$L^p(\M;\{\ell^2_{\Lambda}\}_c)$ and its norm in the latter
space is equal to the norm of $z$ in $L^p(B(\ell^2_\Lambda)\overline{\otimes} \M)$.
Conversely for any $(b_{\lambda})_{\lambda\in\Lambda}$ in 
$L^p(\M;\{\ell^2_{\Lambda}\}_c)$, the matrix $(z_{\lambda,\mu})_{\lambda,\mu\in\Lambda}$
defined, for any $\lambda\in\Lambda$, 
by $z_{\lambda,\mu_0}=b_\lambda$ and  $z_{\lambda,\mu}=0$ if $\mu\not=\mu_0$,
represents an element $z$ of $L^p(B(\ell^2_\Lambda)\overline{\otimes} \M)$. 

Likewise $L^p(\M;\{\ell^2_{\Lambda}\}_r)$ can be identified with any row subspace of 
$L^p(B(\ell^2_\Lambda)\overline{\otimes} \M)$.

We will use the fact that if $p\geq 1$ is finite, then for any 
$(a_{\lambda})_{\lambda\in\Lambda}$ in $L^{p}(\M;\{\ell^2_{\Lambda}\}_r)$
and for any $(b_{\lambda})_{\lambda\in\Lambda}$ in $L^{p}(\M;\{\ell^2_{\Lambda}\}_c)$,
the family $(a_\lambda b_\lambda)_{\lambda\in\Lambda}$ is summable in $L^{\frac{p}{2}}(\M)$
for the usual topology. This allows to define the sums
\begin{equation}\label{Summation}
\sum_{\lambda} a_\lambda b_\lambda,
\qquad \sum_{\lambda} a_\lambda a_\lambda^*
\qquad\hbox{and}\qquad \sum_\lambda b_\lambda^* b_\lambda
\end{equation}
as elements of $L^{\frac{p}{2}}(\M)$.

In the case when $p=\infty$,  the spaces 
$L^\infty(\M;\{\ell^2_{\Lambda}\}_r)$ and 
$L^\infty(\M;\{\ell^2_{\Lambda}\}_c)$ coincide with the row 
space $R_\Lambda^\omega(M)$ and the column space $C_\Lambda^\omega(M)$ from
\cite[1.2.26-1.2.29]{BLM}, respectively. For any  
$(a_{\lambda})_{\lambda\in\Lambda}$ in $L^{\infty}
(\M;\{\ell^2_{\Lambda}\}_r)$
and for any $(b_{\lambda})_{\lambda\in\Lambda}$ in 
$L^{\infty}(\M;\{\ell^2_{\Lambda}\}_c)$,
the family $(a_\lambda b_\lambda)_{\lambda\in\Lambda}$ 
is summable in the $w^*$-topology
of $\M$ and the sums in  (\ref{Summation}) are defined 
in $\M$ according to this topology.

The next lemma is a polar decomposition principle 
which will be used several times in our arguments.
We state it for column valued $L^p(\M)$-spaces; a similar statement holds for 
row valued $L^p(\M)$-spaces.

\begin{lemma}\label{pdc}
Let $1\leq p<\infty$, let $\Lambda$ be an index set and consider a family
$(b_{\lambda})_{\lambda\in\Lambda}$ of $L^p(\M)$.  
The following assertions are equivalent.
\begin{itemize}
\item[(i)] The family $(b_{\lambda})_{\lambda\in\Lambda}$ belongs to $L^p(\M;\{\ell^2_{\Lambda}\}_c)$ and
$\displaystyle{\|(b_{\lambda})_{\lambda}\|_{L^p(\tiny{\M}; \{\ell^2_{\Lambda}\}_c)}\leq 1}$.	
\item[(ii)] There exist a family $(w_{\lambda})_{\lambda\in\Lambda}$ in $L^\infty(\M;\{\ell_{\Lambda}^2\}_c)$ 
and $b$ in $L^p(\M)$ with
$$
\|(w_{\lambda})_{\lambda}\|_{L^\infty(\tiny{\M};\{\ell_{\Lambda}^2\}_c)}\leq
1\qquad \text{and}\qquad\|b\|_p\leq 1,
$$
such that for all $\lambda\in\Lambda$,
$b_{\lambda}=w_{\lambda} b$.
\end{itemize}
\end{lemma}

\begin{proof}
Assume (i). Following the above discussion we fix $\mu_0\in\Lambda$
and consider the element $z\in L^p(B(\ell^2_\Lambda)\overline{\otimes} \M)$
such that $z_{\lambda,\mu_0}=b_\lambda$ and  $z_{\lambda,\mu}=0$ for any $\mu\not=\mu_0$.
Then we have
$$
z=\,\sum_{\lambda} E_{\lambda,\mu_0}\otimes b_\lambda,
$$
with norm convergence in $L^p(B(\ell^2_\Lambda)\overline{\otimes} \M)$.
Consider the polar decomposition $z=w\vert z\vert$ of $z$, with $w\in  B(\ell^2_\Lambda)\overline{\otimes} \M$
and $\vert z\vert \in L^p(B(\ell^2_\Lambda)\overline{\otimes} \M)$.
Then we have 
$$
\vert z \vert =E_{\mu_0,\mu_0}\otimes b,
\qquad\hbox{with}\qquad b=\Bigl(\sum_\lambda b_\lambda^*b_\lambda\Bigr)^{\frac12}.
$$
We note that $\norm{b}_p =\|(b_{\lambda})_{\lambda}\|_{L^p(\tiny{\M}; \{\ell^2_{\Lambda}\}_c)}\leq 1$.

Now if $(w_{\lambda,\mu})_{\lambda,\mu\in\Lambda}$
is the family of $\M$ representing $w$, then for any $\lambda\in\Lambda$,
we have $b_{\lambda}=w_{\lambda,\mu_0}b$ and
$w_{\lambda,\mu}=0$ if $\mu\not=\mu_0$.
Hence 
the family  $(w_{\lambda,\mu_0})_{\lambda\in\Lambda}$ belongs 
to $L^\infty(\M;\{\ell_{\Lambda}^2\}_c)$ 
and its norm in the latter space is $\norm{w}\leq 1$. 
This yields (ii).

The converse implication ``(ii)$\,\Rightarrow\,$ (i)" 
folows from the fact that for any finite $F\subset \Lambda$, we have
$$
\sum_{\lambda \in F} (w_\lambda b)^* (w_\lambda b) = b^*\Bigl(
\sum_{\lambda \in F} w_\lambda^* w_\lambda \Bigr)b.
$$
\end{proof}

\begin{definition}\label{S1}
Let $1\leq p<\infty$. 
We let  $L^p(\M;S^1)$ denote the space of all infinite matrices $[x_{ij}]_{i,j\geq1}$ 
in $L^p(\mathcal{M})$ for which there exist families
$$
(a_{ik})_{i,k\geq1}\in L^{2p}(\mathcal{M};\{\ell^2_{\Ndb^2}\}_r)
\qquad\text{and}\qquad (b_{kj})_{k,j\geq1}\in L^{2p}(\mathcal{M};\{\ell^2_{\Ndb^2}\}_c)
$$ 
such that for all $i,j\geq1$,
$$
x_{ij}=\sum_{k =1}^{\infty} a_{ik}b_{kj}.
$$ 
We equip $L^p(\M;S^1)$ with the following norm, 
\begin{equation}\label{Norm}
\|[x_{ij}]\|_{L^p(\mathcal{M};S^1)}
=\inf\left\{\left\|(a_{ik})_{i,k}\right\|_{L^{2p}(\mathcal{M}; \{\ell^2_{\Ndb^2}\}_r)} 
\left\|(b_{kj})_{k,j}\right\|_{L^{2p}(\mathcal{M}; \{\ell^2_{\Ndb^2}\}_c)}\right\},
\end{equation}
where the infimum is taken over all families $(a_{ik})_{i,k\geq1}$ and $(b_{kj})_{k,j\geq1}$ as above.  
\end{definition}

When applying (\ref{Norm}), we will use the fact that we both have
$$
\left\|(a_{ik})_{i,k}\right\|_{L^{2p}(\mathcal{M}; \{\ell^2_{\Ndb^2}\}_r)} 
=\Bignorm{\sum_{i,k} a_{ik} a_{ik}^*}_p^\frac12
\quad\hbox{and}\quad
\left\|(b_{kj})_{k,j}\right\|_{L^{2p}(\mathcal{M}; \{\ell^2_{\Ndb^2}\}_c)}
=\Bignorm{\sum_{j,k} b_{kj}^* b_{kj}}_p^\frac12.
$$

The above definition is a natural extension
of Junge's spaces $L^p(\M;\ell^1)$ introduced in \cite{J}.
A similar argument as in the proof of \cite[Lemma 3.5]{J} shows that $L^p(\M;S^1)$ is a vector 
space and that (\ref{Norm}) is indeed a norm. Moreover $L^p(\M;S^1)$ endowed with this norm is a Banach space.

For any integer $n\geq 1$, let $\Ndb_n=\{1,\ldots,n\}$. We
let $L^p(\M;S^1_n)$ be the subspace of $L^p(\M;S^1)$ of 
matrices $[x_{ij}]_{i,j\geq1}$ with support in $\Ndb_n\times\Ndb_n$. 
We note that $\bigcup_n L^p(\M;S^1_n)$ is dense in $L^p(\M;S^1)$.

\begin{remark}\label{algebraic}
Identifying a finite matrix $[x_{ij}]_{1\leq i,j\leq n}$ 
of elements in $L^p(\M)$ with the sum $\sum_{i,j=1}^n
x_{ij}\otimes E_{ij}$, we see that at the algebraic level, 
$L^p(\M;S^1_n) = L^p(\M)\otimes S^1_n$.
More generally we have a natural embedding
\begin{equation}\label{Embed}
L^p(\M)\otimes S^1\,\subset\, L^p(\mathcal{M};S^1).
\end{equation}
More precisely, 
consider a matrix $c= [c_{ij}]_{i,j\geq 1}$ in $S^1$ and $x\in L^p(\M)$.
Let $c'=[c'_{ik}]_{i,k\geq 1}$ and 
$c''= [c''_{kj}]_{k,j\geq 1}$ in $S^2$ such that $c'c''=c$ and
let $x',x''\in L^{2p}(\M)$ such that $x'x''=x$. Then 
$(c'_{ik}x')_{i,k\geq 1}$ and $(c''_{kj}x'')_{k,j\geq 1}$ belong
to $L^{2p}(\mathcal{M};\{\ell^2_{\Ndb^2}\}_r)$
and $L^{2p}(\mathcal{M};\{\ell^2_{\Ndb^2}\}_c)$, respectively, and 
$c_{ij}x=\sum_{k} (c'_{ik}x')(c''_{kj}x'')$ for all $i,j\geq 1$. 
Thus $[c_{ij}x]_{i,j\geq 1}$ belongs to $L^p(\mathcal{M};S^1)$. 
Identifying this matrix with $x\otimes c$, this yields (\ref{Embed}). 
It is clear that with this convention, $L^p(\M)\otimes S^1$
is a dense subspace of $L^p(\mathcal{M};S^1)$.
\end{remark}

Lemma \ref{P21} below shows that for elements of $L^p(\M;S^1_n)$, the infimum
in (\ref{Norm}) can be taken over finite families only. This will turn out to be 
very convenient 
in future arguments. To obtain this property we will use a natural connection
between the definition of the norm on $L^p(\M;S^1)$ and decomposable operators.

Let $\A$ and $\B$ be $C^*$-algebras. A linear map $\theta\colon \A\to \B$ is said 
to be decomposable if $\theta$ is a linear combination of completely positive 
maps from $\A$ into $\B$. In this case, $\theta$ may be written as 
$\theta=(\theta_1-\theta_2)+i(\theta_3-\theta_4)$, for four completely positive maps
$\theta_j\colon \A\to \B$. 
Note, for example, that any finite rank operator between $C^*$-algebras is decomposable. 
In \cite{H}, Haagerup introduced a norm
$\decnorm{\,\cdotp}$ on the space of all decomposable maps from $\A$ into $\B$. 
We refer to the latter paper and also to 
\cite[Chap. 11 $\&$ 14]{P3} for basic properties 
of this norm. (This norm is 
given in Remark \ref{DecLp}, however we will not need it explicitly here.)

Let $n\geq 1$ and let $\theta\colon M_n\to \M$ be a linear map. According to 
\cite[Prop. 4.5]{LMM},
$$
\decnorm{\theta}=
\inf\left\{\left\|(v_{ik})_{i,k}
\right\|_{L^{\infty}(\mathcal{M}; \{\ell^2_{\Ndb_n\times\Ndb}\}_r)} 
\left\|(w_{kj})_{k,j}\right\|_{L^{\infty}
(\mathcal{M}; \{\ell^2_{\Ndb\times\Ndb_n}\}_c)}\right\},
$$
where the infimum runs over all families $(v_{ik})_{i,k}$ and $(w_{kj})_{k,j}$ in $\M$
such that $\theta(E_{ij}) = \sum_{k=1}^\infty v_{ik}w_{kj}$ for any $1\leq i,j\leq n$. Applying Lemma \ref{pdc}
and its row counterpart, we deduce that for any linear map 
$u\colon M_n\to L^p(\M)$,
\begin{equation}\label{Mag}
\bignorm{\bigl[u(E_{ij})\bigr]}_{L^p(\mathcal{M};S^1_n)}
=\inf\left\{\norm{a}_{2p}\decnorm{\theta}\norm{b}_{2p}\right\},
\end{equation}
where the infimum runs over all $a,b\in L^{2p}(\M)$ and all
linear maps $\theta\colon M_n\to\M$ such that 
$$
u(s) = a\theta(s)b,\qquad s\in M_n.
$$

We will use Pisier's delta norm $\delta$ on $\M\otimes S^1_n$
introduced in \cite[Chapter 12]{P3}
(see also \cite[Sections 6.4-6.5]{BLM}).  
Given a matrix $[y_{ij}]_{1\leq i,j\leq n}$ of elements in $\M$, consider 
the associated operator $\theta\colon M_n\to \M$ defined by $\theta(E_{ij})
= y_{ij}$ for any $1\leq i,j\leq n$. By \cite[Corollary 12.4]{P3}, we have
$\decnorm{\theta} = \|[y_{ij}]\|_{_{\delta}}$. Combining with
(\ref{Mag}), we deduce that for any matrix $[x_{ij}]_{1\leq i,j\leq n}$ 
of elements in $L^p(\M)$, we have
\begin{align}\label{eqnorm}
\|[x_{ij}]\|_{L^p(\tiny{\M};S^1_n)}=\inf\bigl\{\|a\|_{2p}\|[y_{ij}]\|_{_{\delta}}\|b\|_{2p}\bigr\},
\end{align}
where the infimum is taken over all factorizations of 
$[x_{ij}]$ of the form 
$$
x_{ij}=ay_{ij} b,\qquad 1\leq i,j\leq n,
$$
with $a,b$ in $L^{2p}(\M)$ and $y_{ij}$ in $\M$.

\begin{lemma}\label{P21}
Let $1\leq p<\infty$ and let $n\geq 1$. For any $[x_{ij}]_{1\leq i,j\leq n}$ in $L^p(\M;S^1_n)$, 
the following assertions are equivalent.
\begin{itemize}
\item[(i)] $\|[x_{ij}]\|_{L^p(\tiny{\M};S^1_n)}<1$.
\item[(ii)] There exist an integer $m\geq 1$ and families $(a_{ik})_{1\leq i\leq n,1\leq k\leq m}$
and $(b_{kj})_{1\leq k\leq m,1\leq j\leq n}$ in $L^{2p}(\M)$ such that 
$x_{ij}=\sum_{k=1}^m a_{ik}b_{kj}$, for all $1\leq i,j\leq n$, and  
$$
\Bignorm{\sum_{i=1}^n\sum_{k=1}^m a_{ik}a_{ik}^*}_{p}<1
\qquad \text{and}\qquad 
\Bignorm{\sum_{j=1}^n\sum_{k=1}^m b_{kj}^* b_{kj}}_{p}<1.
$$
\end{itemize} 
\end{lemma}

\begin{proof}
Assume (i), that is, $\|[x_{ij}]\|_{L^p(\tiny{\M};S^1_n)}<1$. 
By (\ref{eqnorm}), there exist a matrix $[y_{ij}]_{1\leq i,j\leq n}$ 
of elements in $\M$ and $a,b\in L^{2p}(\M)$ such that 
$$
\norm{a}_{2p} <1,\qquad \norm{b}_{2p} <1,\qquad \|[y_{ij}]\|_{_{\delta}}<1,
$$
and $x_{ij}=ay_{ij}b$ for all $1\leq i,j\leq n$. According to \cite[Proposition 6.5.2]{BLM}  
there exist $m\geq1$, and families $(v_{ik})_{1\leq i\leq n,1\leq k\leq m}$ 
and $(w_{kj})_{1\leq k\leq m,1\leq j\leq n}$ in $\M$ such that
$y_{ij} = \sum_{k=1}^m v_{ik} w_{kj}$ for any $1\leq i,j\leq n$, and 
$$
\Bignorm{\sum_{i=1}^n\sum_{k=1}^m v_{ik}v_{ik}^*}_{\infty}<1,
\qquad \Bignorm{\sum_{j=1}^n\sum_{k=1}^m w_{kj}^* w_{kj}}_{\infty}<1.
$$
For any $1\leq i,j\leq n$ and any $1\leq k\leq m$, set $a_{ik}=a v_{ik}$
and $b_{kj} = w_{kj}b$. Then 
they satisfy the assertion (ii).

The converse implication ``(ii) $\,\Rightarrow\,$ (i)" is obvious.
\end{proof}

\begin{remark}\label{Positive}
We may naturally identify $L^p(\M;S^1_n)$ with $L^p(M_n\overline{\otimes}\M)$ 
as vector spaces (the norms on these two spaces are however different).
Let $L^p(\M;S^1_n)^+$ be the set of all the
$[x_{ij}]_{1\leq i,j\leq n}\in L^p(\M;S^1_n)$
which belong (under this identification)
to the positive cone $L^p(M_n\overline{\otimes}\M)^+$. For such a matrix, we have
\begin{equation}\label{Positive2}
\|[x_{ij}]\|_{L^p(\tiny{\M};S^1_n)} = \Bignorm{\sum_{i=1}^n x_{ii}}_p.
\end{equation}
Indeed since $[x_{ij}]_{1\leq i,j\leq n}$ belongs to $L^p(\M;S^1_n)^+$, there exist
a matrix $B=[b_{kj}]_{1\leq k,j\leq n}$ of elements in $L^{2p}(\M)$ such that
$[x_{ij}]=B^*B$, which reads
$$
x_{ij}=\sum_{k=1}^n b_{ki}^* b_{kj},\qquad 1\leq i,j\leq n.
$$
Then with $a_{ik}=b_{ki}^*$, we have 
$$
\Bignorm{\sum_{i,k=1}^n a_{ik} a_{ik}^*}_p = 
\Bignorm{\sum_{k,j=1}^n b_{kj}^* b_{kj}}_p
=\Bignorm{\sum_{i=1}^n x_{ii}}_p.
$$
This implies the inequality $\leq$ in (\ref{Positive2}). 

The converse inequality (which is true without any positivity assumption) follows
from the fact that if $x_{ij} = \sum_{k} a_{ik}b_{kj}$ for any $1\leq i,j\leq n$
and some $a_{ik},b_{kj}\in L^{2p}(\M)$, then 
$$
\Bignorm{\sum_{i=1}^n x_{ii}}_p = \Bignorm{\sum_{i,k} a_{ik}b_{ki}}_p
\leq \Bignorm{\sum_{i,k=1}^n a_{ik} a_{ik}^*}_p^{\frac12}
\Bignorm{\sum_{i,k=1}^n b_{ki}^* b_{ki}}_p^{\frac12},
$$
by H\"older's inequality.
\end{remark}

We now establish an injectivity property of the $L^p(\M;S^1_n)$-norms.

\begin{lemma}\label{3.7}
Assume that $e\in\M$ is a projection with finite trace. Let $n\geq1$ be an integer. 
For any matrix $[x_{ij}]_{1\leq i,j\leq n}$ 
of elements in $L^p(e\M e)$, we have 
\begin{align}\label{pis}
\norm{[x_{ij}]}_{L^p(e \tiny{\M} e;S^1_n)}=\norm{[x_{ij}]}_{L^p(\tiny{\M};S^1_n)}.
\end{align} 
\end{lemma}

\begin{proof} The inequality $\geq$ is obvious. To prove the converse, 
it suffices, by density of $e\M e$ in $L^p(e\M e)$, to verify  
the inequality $\leq$ in (\ref{pis}) when each $x_{ij}$
belongs to $e\M e$. Assume this property, along with
$\norm{[x_{ij}]}_{L^p(\tiny{\M};S^1_n)} <1$.

By (\ref{eqnorm}),
there exist $a$ and $b$ in $L^{2p}(\M)$ and a matrix $[y_{ij}]_{1\leq i,j\leq n}$ 
of elements in $\M$ such that 
$x_{ij}=a y_{ij} b$ for any $1\leq i,j\leq n$,
$\|a\|_{2p}<1$, $\|b\|_{2p}<1$ and $\left\|[y_{ij}]\right\|_{_{\delta}}<1$. 
By assumption, $ex_{ij}e=x_{ij}$ hence we actually have
$x_{ij}= ea y_{ij} be$ for any $1\leq i,j\leq n$. 
Using polar decompositions we can write $be=wb'$ 
and $ea=a'v$, 
with $b'=\vert be\vert$, $a'=\vert a^* e\vert$ and
$v,w\in \M$ such that $\norm{v}\leq 1$ and $\norm{w}\leq 1$.
Note that $a',b'\in L^{2p}(e\M e)^+$ and that
$\norm{a'}_{2p}<1$ and $\norm{b'}_{2p}<1$.
It follows from these factorizations that
\begin{align}\label{d1}
x_{ij}=a'v y_{ij} wb',\qquad 1\leq i,j\leq n.
\end{align}

Since $e$ has a finite trace, it belongs to $L^{2p}(\M)$ hence we can choose
$\epsilon>0$ such that 
\begin{equation}\label{Prime}
\norm{a'+\epsilon e }_{2p}<1
\qquad\hbox{and}\qquad
\norm{b'+\epsilon e}_{2p}<1.
\end{equation}
Both $a'+\epsilon e$ and $b'+\epsilon e$ have an inverse in $e\M e$. Then we can define
\begin{equation}\label{Def-z}
z_{ij} =(a'+\epsilon e)^{-1}  x_{ij} (b'+\epsilon e)^{-1},\qquad 1\leq i,j\leq n.
\end{equation}
Since each $x_{ij}$ belongs to $e\M e$, each $z_{ij}$ belongs to $e\M e$ as well.
Further we have 
\begin{equation}\label{Facto}
x_{ij}= (a'+\epsilon e)  z_{ij} (b'+\epsilon e),\qquad 1\leq i,j\leq n.
\end{equation}

Let us now show that 
\begin{equation}\label{z-y}
\left\|[z_{ij}]\right\|_{\delta}
\leq \left\|[y_{ij}]\right\|_{\delta}.
\end{equation}
Here the delta norm on the left-hand side is computed in
$e\M e\otimes S^1_n$ whereas the  delta norm on the right-hand side is computed in
$\M \otimes S^1_n$.
We observe that since $a'\in L^{2p}(e\M e)^+$, $(a'+\epsilon e)^{-1}a'$ belongs to 
$e\M e$ and we have $\norm{(a'+\epsilon e)^{-1}a'}_\infty\leq 1$. Likewise, we have
$\norm{b'(b'+\epsilon e)^{-1}}_\infty\leq 1$. This implies that
\begin{equation}\label{Cont}
\norm{(a'+\epsilon e)^{-1}a'v}_\infty\leq 1
\qquad\hbox{and}\qquad
\norm{wb'(b'+\epsilon e)^{-1}}_\infty\leq 1.
\end{equation}
Let $\theta\colon M_n\to \M$ be the linear map associated with $[y_{ij}]$
and let $\varphi\colon M_n\to e\M e$ be associated with $[z_{ij}]$. By
(\ref{d1}) and (\ref{Def-z}), we have $z_{ij} =(a'+\epsilon e)^{-1}a'v y_{ij} 
wb'(b'+\epsilon e)^{-1}$ for any $1\leq i,j\leq n$. Hence 
$$
\varphi(s) = \bigl[(a'+\epsilon e)^{-1}a'v\bigr]\theta(s) \bigl[wb'
(b'+\epsilon e)^{-1}\bigr],\quad s\in M_n.
$$
It therefore follows from e.g. \cite[(11.4)]{P3} and (\ref{Cont}) that 
$\decnorm{\varphi}\leq \decnorm{\theta}$.

Since $\decnorm{\theta} = \|[y_{ij}]\|_{_{\delta}}$
and $\decnorm{\varphi} = \|[z_{ij}]\|_{_{\delta}}$, by \cite[Corollary 12.4]{P3},
this yields (\ref{z-y}).

Now combining (\ref{Facto}), (\ref{Prime}) and (\ref{z-y}), and using (\ref{eqnorm})
in $L^p(e\M e;S^1_n)$, we obtain that $\bignorm{[x_{ij}]}_{L^p(e\tiny{\M}e;S^1_n)} <1$.
This proves the result.
\end{proof}

For any semifinite and hyperfinite von Neumann algebra $\M$, and for any operator space $E$, 
Pisier \cite[Chapter 3]{P5} introduced a vector valued noncommutative 
$L^p$-space, that we denote by $L^p(\M)[E]$. The next statement shows that Definition \ref{S1}
is consistent with \cite{P5}.

\begin{proposition}\label{Pisier}
Let $\M$ be a semifinite and hyperfinite von Neumann algebra, and let $1\leq p<\infty$.
Equip the spaces $S^1$ and $S^1_n$ with their natural operator space structures (see e.g. \cite[Section 9.3]{ER}).
Then 
$$
L^p(\M;S^1) = L^p(\M)[S^1]
\qquad\hbox{and}\qquad
L^p(\M;S^1_n)=L^p(\M)[S^1_n]
$$
isometrically, for all $n\geq 1$.
\end{proposition}

\begin{proof}
We assume that the semifinite von Neumann algebra
$\M$ is hyperfinite. By density it suffices to prove that
for any $n\geq 1$ and for
any matrix $[x_{ij}]_{1\leq i,j\leq n}$ of elements in $L^p(\M)$, we have
\begin{equation}\label{Iso}
\norm{[x_{ij}]}_{L^p(\tiny{\M};S^1_n)} 
=\norm{[x_{ij}]}_{L^p(\tiny{\M})[S^1_n]}.
\end{equation}

Assume first that $\M$ is finite.
For any matrix $[y_{ij}]_{1\leq i,j\leq n}$ of elements in $\M$, 
let $\norm{[y_{ij}]}_{\rm min}$ denote its norm
in the minimal tensor product $\M\otimes_{\rm min} S^1_n$. 
It follows from 
the definition of $\Lambda_p(E)$ in \cite[p.41]{P5}
and from \cite[Theorem 3.8]{P5} that for any 
matrix $[x_{ij}]_{1\leq i,j\leq n}$ of elements in $L^p(\M)$, we have
$$
\norm{[x_{ij}]}_{L^p(\tiny{\M})[S^1_n]} = 
\inf\bigl\{\norm{a}_{2p} \norm{[y_{ij}]}_{\rm min} \norm{b}_{2p}\bigr\},
$$
where the infimum runs over all $a,b\in L^{2p}(\M)$ and all matrices $[y_{ij}]$
of elements in $\M$ such that $x_{ij}=ay_{ij}b$ for any $1\leq i,j\leq n$.
Since $\M$ is hyperfinite, 
hence injective, we have
$$
\norm{[y_{ij}]}_\delta\,=\, \norm{[y_{ij}]}_{\rm min}
$$
for any such $[y_{ij}]$.
This follows from the fact that if $\theta\colon M_n\to\M$ is the linear map associated
with $[y_{ij}]$, then $\norm{[y_{ij}]}_{\rm min}=\norm{\theta}_{cb}$, 
$\norm{[y_{ij}]}_{\delta}=\norm{\theta}_{dec}$, as mentioned above,
and $\norm{\theta}_{cb}=\norm{\theta}_{dec}$ (see \cite{H}).
Applying (\ref{eqnorm}), we deduce the equality (\ref{Iso}) in that case.

For a possibly non finite $\M$, consider 
$V=\cup e\M e\,$, where the union runs over all projections $e$
in $\M$ with finite trace. The finite case considered above shows that
$$
L^p(e\M e;S^1_n)=L^p(e\M e)[S^1_n]
$$
isometrically, for any such $e$. 
Applying Lemma \ref{3.7} and \cite[Theorem 3.4]{P3}, this implies that
(\ref{Iso}) holds true whenever $x_{ij}\in V$ for all $1\leq i,j\leq n$.  Since
$V$ is dense in $L^p(\M)$, this yields (\ref{Iso}) for any $x_{ij}\in L^p(\M)$.
\end{proof}

In the sequel we consider a second semifinite von Neumann algebra $\N$.
Recall the embedding (\ref{Embed}) from Remark \ref{algebraic}.

\begin{definition}
Let $1\leq p<\infty$ and let 
$T\colon L^p(\mathcal{M})\to L^p(\mathcal{N})$ be a bounded map. 
We say that $T$ is 
\begin{enumerate}[(i)]
\item $S^1$-bounded if $T\otimes I_{S^1}$ extends to a bounded map 
$$
T\overline{\otimes}I_{S^1}\colon L^p(\M;S^1)\longrightarrow L^p(\N;S^1).
$$
In this case, the norm of $T\overline{\otimes}I_{S^1}$ 
is called the $S^1$-bounded norm of $T$ and is denoted by $\|T\|_{S^1}$;
\item $S^1$-contractive if it is $S^1$-bounded and $\|T\|_{S^1}\leq1$.
\end{enumerate}
\end{definition}

\begin{remark}\label{n}
It is plain that $T\colon L^p(\mathcal{M})\to L^p(\mathcal{N})$ is $S^1$-bounded
if and only if there exists 
a constant $K\geq 0$ such that 
$$
\bignorm{T \otimes I_{S^1_n}\colon L^p(\M;S^1_n)\longrightarrow L^p(\N;S^1_n)}\,\leq K
$$
for any $n\geq 1$. In this case, $\|T\|_{S^1}$ is the smallest $K\geq 0$ satisfying this property.
\end{remark}

\begin{remark}\label{p=1}
We have natural isometric identifications 
$$
L^2(\M;\{\ell^2_{\Ndb^2}\}_r)
=L^2(\M;\{\ell^2_{\Ndb^2}\}_c)= L^2(B(\ell^2)\overline{\otimes}\M).
$$
They imply that
$$
L^1(\M;S^1) = L^1(B(\ell^2)\overline{\otimes}\M)
\qquad\hbox{isometrically}.
$$
Consequently, a bounded map $T\colon L^1(\mathcal{M})\to L^1(\mathcal{N})$
is $S^1$-bounded if and only if $T$ is completely bounded and 
$\|T\|_{S^1}=\norm{T}_{cb}$ in this case.
\end{remark}

Assume that $\M,\N$ are two semifinite and hyperfinite von Neumann algebras,
and let $T\colon L^p(\mathcal{M})\to L^p(\mathcal{N})$ be a bounded map.
We say that $T$ is 
completely regular if there 
exists a constant $K\geq 0$ such that for any $n\geq 1$,
$$
\bignorm{T\otimes I_{M_n}\colon L^p(\mathcal{M})[M_n]
\longrightarrow  L^p(\mathcal{N})[M_n]}\,\leq K.
$$
In this case, the completely regular
norm $\norm{T}_{reg}$ is defined as 
the least possible $K$ satisfying this property.
This concept was introduced in \cite{P2}. 
It is shown in the latter paper that if $T$ is 
completely regular, then for any operator space 
$E$, $T\otimes I_E$ extends to a bounded operator $T\overline{\otimes} I_E$
from $L^p(\mathcal{M})[E]$ into $L^p(\mathcal{N})[E]$, with 
\begin{equation}\label{Reg-E}
\bignorm{T\overline{\otimes} I_E\colon L^p(\mathcal{M})[E]\longrightarrow  
L^p(\mathcal{N})[E]}\,\leq \norm{T}_{reg}.
\end{equation}
We refer to \cite{JR} and \cite{AK} for developments and further results.

\begin{proposition}\label{3.8}
Suppose that $\M$ and $\N$ are semifinite and hyperfinite 
von Neumann algebras, 
let $1\leq p<\infty$ and let 
$T\colon L^p(\M)\to L^p(\N)$ be a bounded operator. Then 
$T$ is $S^1$-bounded if and only if $T$ is completely regular and
in this case, we have $\|T\|_{S^1}=\norm{T}_{reg}$.
\end{proposition}

\begin{proof}
Suppose that $T$ is $S^1$-contractive. 
By Proposition \ref{Pisier}, we have 
\begin{align}\label{eqp1}
\|T\otimes I_{S^1_n}\colon L^p(\M)[S^1_n]\longrightarrow L^p(\N)[S^1_n]\|\leq \|T\|_{S^1}.
\end{align}
for every $n\geq 1$.
Assume that $p>1$ and let $p'=p/(p-1)$ be the conjugate number of $p$.  
By \cite[Theorem 4.1]{P5}, we both have 
$$
\bigl(L^p(\M)[S^1_n]\bigr)^*\cong L^{p^{\prime}}(\M)[M_n] \qquad\text{and}\qquad 
\bigl(L^p(\N)[S^1_n]\bigr)^*\cong L^{p^{\prime}}(\N)[M_n].
$$
isometrically.
Passing to the adjoint in (\ref{eqp1}), we obtain that
$$
\|T^*\otimes I_{M_n}\colon L^{p^{\prime}}(\N)[M_n]\longrightarrow 
L^{p^{\prime}}(\M)[M_n]\|\leq \|T\|_{S^1}
$$
for every $n\geq 1$.
Thus $T^*$ is completely regular, with $\norm{T^*}_{reg}\leq \|T\|_{S^1}$.
It now follows from \cite[Lemma 2.3]{P2} that $T$ is completely regular as well,
with $\norm{T}_{reg}\leq \|T\|_{S^1}$. The case $p=1$ is proved similarly, 
using Remark \ref{p=1}.

The converse is clear, using Proposition \ref{Pisier} again.
\end{proof}

\begin{remark}\label{l1}
Junge's space $L^p(\M;\ell^1)$ from \cite{J} coincides
with the subspace of $L^p(\M;S^1)$
of matrices $[x_{ij}]_{i,j\geq 1}$ such that $x_{ij}=0$ for any $i\not= j$. 
In \cite[Definition 2.5]{LMZ}, we introduced $\ell^1$-boundedness
by saying that a bounded map
$T\colon L^p(\mathcal{M})\to L^p(\mathcal{N})$ is $\ell^1$-bounded if 
$T\otimes I_{\ell^1}$ extends to a bounded map 
from $L^p(\M;\ell^1)$ into $L^p(\N;\ell^1)$. 
It is plain that any $S^1$-bounded map $T$ is $\ell^1$-bounded, with $\norm{T}_{\ell^1}
\leq \norm{T}_{S^1}$. However 
\cite[Example 2.7]{LMZ} shows that the converse is not true. 
\end{remark}

We now state the main result of this section.

\begin{theorem}\label{ls1}
Suppose that $T\colon L^p(\M)\to L^p(\N)$ is a completely positive operator.
Then $T$ is $S^1$-bounded and $\|T\|_{S^1}=\|T\|$.	
\end{theorem}

\begin{proof}
Let $T\colon L^p(\M)\to L^p(\N)$ be a completely positive operator. 
Fix some $n\geq1$. 
Let $x=[x_{ij}]_{1\leq i,j\leq n}$
be a matrix of elements in $L^p(\M)$ with 
$\|x\|_{L^p(\tiny{\M};S_n^1)} <1$. According to Lemma \ref{P21}, 
there exist an integer $m\geq 1$
and families $(a_{ik})_{1\leq i\leq n,1\leq k\leq m}$ 
and $(b_{kj})_{1\leq k\leq m,1\leq j\leq n}$ in $L^{2p}(\M)$ such that 
$$
\Bignorm{\sum_{i,k} a_{ik}a_{ik}^*}_p<1,
\qquad
\Bignorm{\sum_{k,j=1} b_{kj}^*b_{kj}}_p<1
\qquad\text{and}\qquad
x_{ij}=\sum _{k=1}^m a_{ik}b_{kj}
$$ 
for any $1\leq i,j\leq n$. We introduce 
$$
r_{ij} =\sum_{k=1}^m a_{ik}a_{jk}^*,
\qquad
s_{ij}=\sum_{k=1}^m b_{ki}^* b_{kj}
\qquad\hbox{and}\qquad
z_{ij}=\begin{pmatrix}
r_{ij} & x_{ij}\\
x_{ji}^* & s_{ij}
\end{pmatrix}
$$
for any $1\leq i,j\leq n$. 
Then we set 
$$ 
r=[r_{ij}],
\qquad s=[s_{ij}]
\qquad\hbox{and}\qquad
z=[z_{ij}].
$$
With $x^* = [x^*_{ji}]$, we  may write
$$
z= \begin{pmatrix}
r & x \\
x^*& s
\end{pmatrix}.
$$
Following Remark \ref{Positive} and (\ref{Spn}) 
we regard $x,x^*,r,s$ as elements of $S^{p}_{n}\otimes L^p(\M)
=L^p(M_{n}\overline{\otimes}\M)$ 
and we regard $z$ as an element of $S^{p}_{2n}\otimes L^p(\M) = L^p(M_{2n}\overline{\otimes}\M)$.

Now consider $a=[a_{ik}]_{1\leq i\leq n, 1\leq k\leq m}$ and 
$b=[b_{kj}]_{1\leq k\leq m, 1\leq j\leq n}$, regarded as elements of
$S^{p}_{n,m}\otimes L^p(\M)$ and $S^{p}_{m,n}\otimes L^p(\M)$, respectively,
and let $c=\begin{pmatrix} a \\ b^* \end{pmatrix} \in S^{p}_{2n,m}\otimes L^p(\M)$. 
It follows from the above definitions that 
$$
z= \begin{pmatrix} a\\ b^* \end{pmatrix} 
\begin{pmatrix} a^* & b \end{pmatrix}\, =cc^*,
$$
hence 
$z\in L^{p}(M_{2n}\overline{\otimes}\M)^+$.

Let us write $T_n = I_{S^p_{n}}\otimes T$ for simplicity. By assumption,
$T_{2n}$ is positive hence 
$$
T_{2n}(z) =	\,
\begin{pmatrix}
T_n(r)&T_n(x)\\ T_n(x^*)&T_n(s)
\end{pmatrix}\,\in\, L^p(M_{2n}\overline{\otimes}\M)^+.
$$
Consider the positive square root $(T_{2n}(z))^{1/2}$, which belongs to 
$L^{2p}(M_{2n} \overline{\otimes}\M)^+$. We may write it as
$$
(T_{2n}(z))^{1/2} 
=\begin{pmatrix}
\alpha &\beta\\
\beta^*&\delta
\end{pmatrix},
$$
with $\alpha$, $\beta$, $\delta$ in $L^{2p}(M_n \overline{\otimes} \M)$, 
and $\alpha\geq0$, $\delta\geq0$. Then,
\begin{align}
\label{r} T_n(r)&=\alpha^2+\beta\beta^*;\\ 
\label{s} T_n(s)&=\beta^*\beta+\delta^2;\\
\label{x} T_n(x)&=\alpha\beta+\beta\delta.
\end{align}
Write $\alpha=[\alpha_{ij}]$, $\beta=[\beta_{ij}]$
and $\delta=[\delta_{ij}]$.  Using (\ref{x}), we have
$$
T(x_{ij})=\sum_{k=1}^n\alpha_{ik}\beta_{kj}+
\sum_{k=1}^n\beta_{ik}\delta_{kj},
\quad 1\leq i,j\leq n.
$$
Let
\begin{align*}
a =\left(\sum_{i,k}\alpha_{ik}\alpha_{ik}^*+\sum_{i,k}\beta_{ik}\beta_{ik}^*\right)^{1/2}\	
\qquad\hbox{and}\qquad
b=\left(\sum_{j,k}\beta_{kj}^*\beta_{kj}+\sum_{k,j}\delta_{kj}^*\delta_{kj}\right)^{1/2},
\end{align*} 
then the above factorization implies that  
$$
\|[T(x_{ij})]\|_{L^p(\tiny{\N};S^1_n)}\leq\|a\|_{2p}\|b\|_{2p}.
$$
Now observe that by (\ref{r}) and (\ref{s}), and the fact that $\alpha^*=\alpha$
and $\delta^*=\delta$, we have
$$
T(r_{ii}) = \sum_{k}\alpha_{ik}\alpha_{ik}^* + \beta_{ik}\beta_{ik}^*
\qquad\hbox{and}\qquad
T(s_{jj}) = \sum_{k}\beta_{kj}^*\beta_{kj} +\delta_{kj}^*\delta_{kj}
$$
for any $1\leq i,j\leq n$.
Consequently,
\begin{align*}
\|a\|_{2p}^{2p} &=\Bignorm{\sum_{i,k}\alpha_{ik}
\alpha_{ik}^*+\sum_{i,k}\beta_{ik}\beta_{ik}^*}_p^p\\
&=\Bignorm{\sum_i T(r_{ii})}_p^p\\
& \leq\|T\|^p\Bignorm{\sum_i r_{ii}}^p_p
= \|T\|^p\Bignorm{\sum_{i,k}a_{ik}a_{ik}^*} _p^p\leq\|T\|^p.
\end{align*}
Similarly, we can show that $\|b\|_{2p}^{2p}\leq\|T\|^p$, and therefore
$$
\|[T(x_{ij})]\|_{L^p(\tiny{\N};S^1_n)}\leq\|T\|.
$$
The result follows at once.
\end{proof}

\begin{remark}\label{DecLp}
Let $1\leq p\leq \infty$.
Following \cite{JR, AK} we say that a bounded map
$T\colon L^p(\M)\to L^p(\N)$ is decomposable 
if there exist two bounded maps 
$S_1,S_2 \colon L^p(\M)\to L^p(\N)$ such that 
the mapping 
$$
\Gamma_{S_1,S_2} : = \begin{pmatrix} S_1 & T \\ T_* & S_2\end{pmatrix}\colon L^p(M_2\overline{\otimes}\M)
\longrightarrow L^p(M_2\overline{\otimes}\N)
$$
taking any $\begin{pmatrix} r & x \\ y & s\end{pmatrix}$ to 
$\begin{pmatrix} S_1(r) & T(x) \\ T(y^*)^* & S_2(s)\end{pmatrix}$,
with $x,y,r,s \in L^p(\M)$, is completely posivive. This is equivalent to
$T$ being a linear combination of completely positive maps $L^p(\M)\to L^p(\N)$.
In this case, the decomposable norm of $T$ 
is defined by 
\begin{equation}\label{DECNORM}
\norm{T}_{dec}=\inf\bigl\{\max\{\norm{S_1},\norm{S_2}\}\bigr\},
\end{equation}
where the infimum
is taken over all possible pairs $(S_1,S_2)$ such that $\Gamma_{S_1,S_2}$ is completely positive . 
When $T$ is completely positive, $\Gamma_{T,T}$ is completely positive
and we have $\norm{T}_{dec}=\norm{T}$ in this case.

With these definitions in mind,
it is clear from the proof of Theorem \ref{ls1} that the latter generalizes as follows,
for any $1\leq p<\infty$:
\begin{equation}\label{Dec-case}
\hbox{Any decomposable map}\ T\colon L^p(\M)\to L^p(\N)\ \hbox{is}\  S^1\hbox{-bounded, with}\ 
\norm{T}_{S^1}\leq \norm{T}_{dec}.
\end{equation}

In the special case when $\M,\N$ are hyperfinite, the converse is true, that is, any 
$S^1$-bounded map $T\colon L^p(\M)\to L^p(\N)$ is decomposable, with 
$\norm{T}_{S^1}\leq \norm{T}_{dec}$. This follows from Proposition \ref{3.8} and 
\cite[Theorem 3.23]{AK}. We do not know 
if this property is true for general semifinite von Neumann algebras.
\end{remark}

We finally mention that 
Proposition \ref{3.8} together with \cite[Proposition 2.2]{P2} show that when 
$\M$, $\N$ are semifinite and hyperfinite 
von Neumann algebras, every $S^1$-bounded operator 
$T\colon L^p(\M)\to L^p(\N)$
is completely bounded. We do not know whether this property
is true for general semifinite von Neumann algebras,
except in the trivial cases 
$p=2$ and $p=1$ (see Remark \ref{p=1}).

\section{Separating maps with a direct Yeadon type factorization}
The notion of Yeadon type factorization was introduced in \cite{LMZ}, 
in reference to Yeadon's characterization of isometries on noncommutative
$L^p$-spaces for $1\leq p\not=2<\infty$ \cite{Y}.
In this section, we introduce the notion
of direct Yeadon type factorization and we discuss 
the relationship between the 
norm, the completely bounded norm and the $S^1$-bounded norm
of operators which admit such a factorization.

First we recall some prerequisite concepts and results. 
A Jordan homomorphism between von Neumann algebras $\M$ and $\N$ is a 
linear map $J\colon\M\to\N$ that preserves involution and the Jordan 
product $(x,y)\mapsto \frac{1}{2}(xy+yx)$. The interested reader is 
referred to \cite[Chapter 7]{HS}, \cite{S} and 
\cite[Exercises 10.5.21-10.5.31]{KR2} 
for information on these maps. 
We note for further use that any Jordan homomorphism is positive.

We assume that $(\mathcal{M},\tau_{\tiny{\M}})$ and $(\mathcal{N},\tau_{\tiny{\N}})$ 
are semifinite
and we let $1\leq p<\infty$. Following \cite{LMZ}, 
we say that an operator $T\colon L^p(\mathcal{M})\to L^p(\mathcal{N})$ 
has a Yeadon type factorization if there exist a $w^*$-continuous
Jordan homomorphism $J\colon\mathcal{M}\to\mathcal{N}$, 
a partial isometry $w\in\mathcal{N}$, 
and a positive operator $B$ affiliated with $\mathcal{N}$, 
which satisfy the following conditions:
\begin{enumerate}[(a)]
\item $w^{\ast}w=J(1)=s(B)$, the support projection of $B$;
\item  every spectral projection of $B$ commutes with $J(x)$, for all $x\in\mathcal{M}$;
\item $T(x)=w BJ(x)$ for all $x\in\mathcal{M}\bigcap L^p(\mathcal{M})$.
\end{enumerate}
In this case, $w$, $B$ and $J$  
are uniquely determined by $T$
and we call $(w, B, J)$ the Yeadon triple associated with $T$.

Yeadon's Theorem \cite{Y} asserts that if $p\neq2$, any isometry
$T\colon L^p(\mathcal{M})\to L^p(\mathcal{N})$ admits a Yeadon type factorization.

Following \cite{LMZ}, 
we say that an operator $T\colon L^p(\mathcal{M})\to L^p(\mathcal{N})$ 
is separating if it preserves disjointness of elements; that is, if 
for $x,y\in L^p(\mathcal{M})$ such that $x^*y=xy^*=0$, then we have
$T(x)^*T(y)=T(x)T(y)^*=0$. 
It is shown in \cite{HRW,LMZ} that $T$ admits a Yeadon type factorization 
if and only if it is separating.

Let $J:\M\to\N$ be a Jordan homomorphism and let $\D\subset\N$ be the
$W^*$-algebra generated by $J(\M)$. Then $J(1)$ is the unit of $\D$. By e.g.
\cite[Theorem 3.3]{S}, there exist projections 
$e$ and $f$ in the center of $\D$ such that
\begin{enumerate}[(i)]
\item $e+f=J(1)$.
\item $x\mapsto J(x)e$ is a $*$-homomorphism.
\item $x\mapsto J(x)f$ is an anti-$*$-homomorphism.
\end{enumerate}
Let $\N_1=e\N e$ and $\N_2=f\N f$. 
We let $\pi\colon \M\to\N_1$ and $\sigma\colon\M\to\N_2$ be defined by 
$\pi(x)=J(x)e$ and $\sigma(x)=J(x)f$, for all $x\in\M$. Then 
$J$ is valued in 
$\N_1\mathop{\oplus}\limits^{\infty}\N_2$
and $J(x)=\pi(x)+\sigma(x)$, for all $x\in\M$. As in \cite{LMZ}
we use the notations 
\begin{align}\label{cd}
J=\begin{pmatrix}
\pi&0\\
0&\sigma
\end{pmatrix}\quad\text{ and }\quad J(x)=\begin{pmatrix}
\pi(x)&0\\
0&\sigma(x)
\end{pmatrix}	
\end{align}
to refer to such a decomposition.

\begin{definition}\label{DYTF}
We say that a separating map $T\colon L^p(\mathcal{M})\to L^p(\mathcal{N})$
admits a direct (resp. anti-direct) 
Yeadon type factorization if the Jordan homomorphism of its
Yeadon triple is a $\ast$-homomorphism (resp. an anti-$\ast$-homomorphism).
\end{definition}

The above definition is partly motivated by a result due to 
Junge-Ruan-Sherman \cite[Proposition 3.2]{JRS} which asserts that
if $p\neq2$, an isometry
$T\colon L^p(\mathcal{M})\to L^p(\mathcal{N})$ admits a direct 
Yeadon type factorization if and only if $T$ is a $2$-isometry, if 
and only if $T$ is a complete isometry.

\begin{remark}\label{Tensorization}
Let
$T\colon L^p(\mathcal{M})\to L^p(\mathcal{N})$
be a separating map and let $(w,B,J)$ be its Yeadon triple.

\smallskip
(a)$\,$
The mapping $w^*T(\,\cdotp)$, which maps any $x\in \M\cap L^p(\M)$ 
to $BJ(x)$, is also a separating map. Its Yeadon triple is 
$(J(1), B, J)$.  Since $J$ is positive, $B$ is positive and $B$ commutes with
the range of $J$, the mapping 
$w^*T(\,\cdotp)$ is positive.

\smallskip
(b)$\,$ Assume that $J=\pi$ is a $\ast$-homomorphism,
so that $T$ has a direct 
Yeadon type factorization.
For any $n\geq 1$, $I_{M_n}\otimes \pi$ is a
$\ast$-homomorphism from $M_n\overline{\otimes}\mathcal{M}$
into $M_n\overline{\otimes}\mathcal{N}$. Hence 
$I_{S^p_n}\otimes T\colon L^p(M_n\overline{\otimes}\mathcal{M})\to 
L^p(M_n\overline{\otimes}\mathcal{N})$ admits a Yeadon type factorization.
Indeed the 
Yeadon triple of $I_{S^p_n}\otimes T$ is equal to $(I_n\otimes w, I_n\otimes B,
I_{M_n}\otimes \pi)$. It follows from (a) that in this case, 
$w^*T(\,\cdotp)$ is  completely positive.
\end{remark}

\begin{remark}\label{cdec}
Let $T\colon L^p(\M)\to L^p(\N)$ be a separating 
operator, with Yeadon triple $(w,B,J)$. 
Assume that $w=J(1)$, so that 
$$
T(x)=BJ(x),\qquad
x\in \M\cap L^p(\M).
$$
Consider a decomposition of $J$ as in (\ref{cd}). This induces a 
direct/anti-direct decomposition of $T$, as follows.

Recall
$\N_1=e\N e$ and $\N_2=f\N f$. Then $\N_1,\N_2$ are semifinite and
we have
$$
L^p(\N_1)\mathop{\oplus}\limits^p L^p(\N_2)
\,=\,L^p(\N_1\mathop{\oplus}\limits^{\infty}\N_2)
\subset L^p(\N).
$$
Set $B_1= Be$ and $B_2=Bf$. Since $B=BJ(1)$, we have $B=B_1+B_2$.
Moreover $B$ commutes with
the range of $J$, that is, $B$ is affiliated with $J(\M)'$. This implies that
$B$ commutes with $e$ and $f$. Consequently,
$B_1$ is affiliated with $\N_1$ and $B_2$ is affiliated with $\N_2$.
Now define
$$
T_1\colon L^p(\M)\longrightarrow L^p(\N_1)
\qquad\hbox{and}\qquad
T_2\colon L^p(\M)\longrightarrow L^p(\N_2)
$$
by setting 
$$
T_1(x)= T(x)e
\qquad\hbox{and}\qquad
T_2(x)= T(x)f,
\quad
x\in L^p(\M).
$$
Then
$$
T=T_1+T_2.
$$ 
Further
$T_1$ is a separating operator and its Yeadon triple is
equal to $(1_{\tiny{\N}_1}, B_1, \pi)$. Likewise
$T_2$ is a separating operator  and its Yeadon triple is
equal to $(1_{\tiny{\N}_2}, B_2, \sigma)$. In particular,
$T_1$ has a direct Yeadon type factorization
whereas $T_2$ has an anti-direct Yeadon type factorization.

In the case when $w\not=J(1)$, one can apply the following decomposition
principle to the mapping $w^*T(\,\cdotp)$ from Remark \ref{Tensorization} (a).
\end{remark}

\begin{proposition}\label{4.4}
Let $T\colon L^p(\M)\to L^p(\N)$ be a bounded operator with 
a direct Yeadon type factorization. Then $T$ is completely 
bounded and $\|T\|_{cb}=\|T\|$. 
\end{proposition}

\begin{proof}
Suppose that $T$ has a direct Yeadon type factorization,
with Yeadon triple $(w,B,\pi)$ and
fix some integer $n\geq 1$. Set $\pi_n= I_{M_n}\otimes \pi$,
$w_n = I_n\otimes w$ and $B_n= I_n\otimes B$.
By Remark \ref{Tensorization} (b),
$I_{S^p_n}\otimes T$ is separating with Yeadon triple equal
to $(w_n,B_n,\pi_n)$.

We note that for any $x\in \M\cap L^p(\M)$, we have
$\vert T(x)\vert^p = B^p \pi(\vert x\vert^p)$, hence 
\begin{equation}\label{Trace1}
\norm{T(x)}_p^p\,=\, \tau_{\tiny{\N}}\bigl(B^p \pi(\vert x\vert^p)\bigr).
\end{equation}

Let $y\in (M_n\overline{\otimes}\M)
\cap L^p(M_n\overline{\otimes}\M)$. Then similarly we have
$$
\|(I_{S^p_n}\otimes T)(y)\|_p^p\,=\,\tau_{\tiny{\N,n}}
\bigl(B_n^p \pi_n(\lvert y\rvert^p)\bigr).
$$
Write $x=\lvert y\rvert^p$ and decompose it as $x=[x_{ij}]_{1\leq i,j\leq n}$. Then
$$
\tau_{\tiny{\N,n}}\bigl(B_n^p
\pi_n(\lvert y\rvert^p)\bigr)=\sum_{i=1}^n\tau_{\tiny\N}\bigl(B^p
\pi(x_{ii})\bigr).
$$
For any $1\leq i\leq n$, we have
$$
\tau_{\tiny\N}\bigl(B^p \pi(x_{ii})\bigr) 
=\norm{T(x_{ii}^{\frac{1}{p}})}_p^p\,
\,\leq \norm{T}^p\norm{x_{ii}^{\frac{1}{p}}}_p^p
\,= \norm{T}^p \tau_{\tiny\M}(x_{ii}),
$$
by (\ref{Trace1}). We infer that 
$$
\|(I_{S^p_n}\otimes T)(y)\|_p^p \leq \norm{T}^p\,\sum_{i=1}^n
\tau_{\tiny\M}(x_{ii})\, =  
\norm{T}^p\,\tau_{\tiny{\M,n}}(x).
$$
This yields $\|(I_{S^p_n}\otimes T)(y)\|_p\,\leq \,\norm{T} \norm{y}_p$, which
proves that $T$ is completely bounded, with $\|T\|_{cb}=\|T\|$. 
\end{proof}

\begin{proposition}\label{4.2}
Let $T\colon L^p(\M)\to L^p(\N)$ be a bounded operator with 
a direct Yeadon type factorization. Then $T$ is $S^1$-bounded and $\|T\|_{S^1}=\|T\|$.
\end{proposition}

\begin{proof}
Suppose that $T$ has a direct Yeadon type factorization,
with Yeadon triple $(w,B,\pi)$. By Remark \ref{Tensorization} (b),
$U : = w^*T(\,\cdotp)$ is completely positive. Hence by 
Theorem \ref{ls1}, $U$ is $S^1$-bounded, 
with $\|U\|_{S^1}=\|U\|$.
Since $wU(x)=T(x)$ for any $x\in L^p(\M)$, this immediately
implies that $T$ is also $S^1$-bounded, with $\|T\|_{S^1}=\|U\|_{S^1}$.
Further we have $\|T\|=\|U\|$, which yields the result.
\end{proof}

In the case when $\M,\N$ are hyperfinite, it follows
from \cite{P2, AK} that any completely positive 
map $T\colon L^p(\M)\to L^p(\N)$ is automatically
completely bounded, with $\norm{T}_{cb}=\norm{T}$.
We do not know if this holds true in general.
If this were true, Proposition \ref{4.4} would be a 
direct consequence of Remark \ref{Tensorization} (b).

\section{Direct Yeadon type factorization and isometries}
We proved in the previous section (Propositions  \ref{4.4} and
\ref{4.2}) that if a contraction $T\colon L^p(\M)
\to L^p(\N)$ admits a direct Yeadon type factorization,
then it is both completely contractive and $S^1$-contractive.
The purpose of this section is to establish converse statements
for isometries. Namely we will show that an isometry $T\colon L^p(\M)
\to L^p(\N)$ admits a direct Yeadon type factorization
provided that either $T$ is completely contractive and $p\not=2$,
or $T$ is $S^1$-contractive.

We need three preparatory lemmas.

\begin{lemma}\label{jan52}
Let $1\leq p<\infty$. 
Let $\M$ and $\N$ be semifinite von Neumann algebras 
and let $b\in L^p(\N)$. Consider a
matrix $[x_{ij}]_{1\leq i,j\leq n}$ of elements in $L^p(\M)$.
We have
\begin{align}\label{ineqs1}
\|[x_{ij}\otimes b]\|_{L^p(\tiny{\M\overline{\otimes}\N};S^1_n)}
\,=\,\norm{b}_p\,\|[x_{ij}]\|_{L^p(\tiny{\M};S^1_n)}.	
\end{align}
\end{lemma}

\begin{proof} The case $p=1$ follows from 
Remark \ref{p=1}, so we may assume that 
$p\not=1$. Let $p'=\frac{p}{p-1}$ be the conjugate number of $p$. 
Let $b\in L^p(\N)$ and let
$c\in L^{p'}(\N)$ such that $\norm{c}_{p'}=1$ and
$\tau_{\tiny{\N}}(bc)=\norm{b}_p$. 
Define
$$
T\colon L^{p^{\prime}}(\M)\to 
L^{p^{\prime}}(\M\overline{\otimes}\N),\quad  
T(z)=z\otimes c.
$$
We claim that $T$ is decomposable, with $\norm{T}_{dec}\leq 1$,
see (\ref{DECNORM}) for the definition.
To check this, consider the polar decomposition
$c=u\vert c\vert$ of $c$. Then $\vert c^*\vert =u\vert c\vert u^*$.
In the space $L^{p'}(M_2\overline{\otimes} \N)$,
the matrix $\begin{pmatrix} \vert c\vert 
&\vert c\vert \\ \vert c\vert & \vert c\vert 
\end{pmatrix} = 
\begin{pmatrix} 1 &1 \\ 1 & 1
\end{pmatrix}\otimes \vert c \vert$ is positive, hence 
$$
C:=\, \begin{pmatrix} \vert c^*\vert & c \\ c^* & \vert c\vert
\end{pmatrix}
= \begin{pmatrix} u & 0 \\ 0 & 1
\end{pmatrix}
\begin{pmatrix} \vert c\vert &\vert c\vert \\ \vert c\vert & \vert c\vert
\end{pmatrix}
\begin{pmatrix} 1 & 0 \\ 0 & u^*
\end{pmatrix}\,\geq 0.
$$
Consequently the operator 
$$
L^{p'}(M_2\overline{\otimes}\M)\longrightarrow
L^{p'}(M_2\overline{\otimes}\M)\otimes L^{p'}(M_2\overline{\otimes}\N)
\,\subset\,L^{p'}(M_4\overline{\otimes}\M\overline{\otimes}\N)
$$
taking $X$ to $X\otimes C$ for any $X\in L^{p'}(M_2\overline{\otimes}\M)$
is completely positive. For any $r,x,y,s$
in $L^{p'}(\M)$, and $X=
\begin{pmatrix} r & x \\ y & s\end{pmatrix}$, the matrix 
$\begin{pmatrix} r\otimes \vert c^*\vert & x\otimes c \\ y\otimes c^*
& s\otimes \vert c\vert\end{pmatrix}$ is an extracted square matrix
of $X\otimes C$. We deduce that the mapping
$\Gamma\colon L^{p'}(M_2\overline{\otimes}\M)\to L^{p'}(M_2\overline{\otimes}\M\overline{\otimes}\N)$
defined by
$$
\Gamma\begin{pmatrix} r & x \\ y & s\end{pmatrix}\, =\, 
\begin{pmatrix} r\otimes \vert c^*\vert & x\otimes c \\ y\otimes c^*
& s\otimes \vert c\vert\end{pmatrix},
\qquad r,x,y,s\in L^{p'}(\M),
$$ 
is completely positive.
Since $r\mapsto r\otimes \vert c^*\vert$ and 
$s\mapsto s\otimes \vert c\vert$ are contractive from $L^{p'}(\M)$
into $L^{p'}(\M\overline{\otimes}\N)$, 
this proves the claim. 

Next the adjoint 
$T^*\colon L^p(\M\overline{\otimes}\N)\to 
L^p(\M)$ is also decomposable, with $\norm{T^*}_{dec}\leq 1$. 
By (\ref{Dec-case}), this implies that $T^*$ is $S^1$-contractive. 
The inequality $\geq$ in
(\ref{ineqs1}) follows since for any $x\in L^p(\M)$,
we have $T^*(x\otimes b)=\norm{b}_p\, x$. The reverse inequality 
$\leq$ in (\ref{ineqs1}) is immediate from the definitions. 
\end{proof}

The next result extends (\ref{DirectSum})
to $S^1$-valued spaces.

\begin{lemma}\label{Sum-S1}
Let $1\leq p<\infty$ and 
let $\N_1$ and $\N_2$ be semifinite von Neumann algebras. For any $n\geq 1$,
for any $[x_{ij}^1]_{1\leq i,j\leq n}$ in $L^p(\N_1;S^1_n)$ and for any 
$[x_{ij}^2]_{1\leq i,j\leq n}$ in $L^p(\N_2;S^1_n)$, we have
\begin{equation}\label{Sum-S1-bis}
\norm{[x_{ij}^1,x_{ij}^2]}_{L^p(\tiny{\N_1}
\mathop{\oplus}\limits^{\infty}\tiny{\N_2}
;S^1_n)}\,=\,\Bigl(
\norm{[x_{ij}^1]}_{L^p(\tiny{\N_1}
;S^1_n)}^p\,+\,\norm{[x_{ij}^2]}_{L^p(\tiny{\N_2}
;S^1_n)}^p \Bigr)^{\frac{1}{p}}.
\end{equation}
\end{lemma}

\begin{proof}
Let $\varepsilon>0$. By Lemma \ref{P21}, there exist an integer $m\geq 1$,
families $[a^1_{ik}]_{1\leq i\leq n,1\leq k\leq m}$ and 
$[b^1_{kj}]_{1\leq k\leq m,1\leq j\leq n}$ in $L^{2p}(\N_1)$, and
families $[a^2_{ik}]_{1\leq i\leq n,1\leq k\leq m}$ and 
$[b^2_{kj}]_{1\leq k\leq m,1\leq j\leq n}$ in $L^{2p}(\N_2)$
such that we have $x_{ij}^1=\sum_k a_{ik}^1 b_{kj}^1$ and 
$x_{ij}^2=\sum_k a_{ik}^2 b_{kj}^2$ for all $1\leq i,j\leq n$, as
well as norm estimates
$$
\norm{(a^1_{ik})_{i,k}}_{L^{2p}(\tiny{\N_1};\{\ell^2_{nm}\}_r)}
\,=\,\norm{(b^1_{kj})_{kj}}_{L^{2p}(\tiny{\N_1};\{\ell^2_{mn}\}_c)}
\,\leq\bigl(\norm{[x_{ij}^1]}_{L^p(\tiny{\N_1}
;S^1_n)}+\varepsilon\bigr)^\frac12
$$
and
$$
\norm{(a^2_{ik})_{i,k}}_{L^{2p}(\tiny{\N_2};\{\ell^2_{nm}\}_r)}
\,=\,\norm{(b^2_{kj})_{kj}}_{L^{2p}(\tiny{\N_2};\{\ell^2_{mn}\}_c)}
\,\leq\bigl(\norm{[x_{ij}^2]}_{L^p(\tiny{\N_2}
;S^1_n)}+\varepsilon\bigr)^\frac12
$$
Let $\N=\N_1 \mathop{\oplus}\limits^{\infty} \N_2$.
Using (\ref{DirectSum}), we have
\begin{align*}
\norm{(a^1_{ik}, a^2_{ik})_{i,k}}_{L^{2p}(\tiny{\N};\{\ell^2_{nm}\}_r)}\,
&=\,\Bignorm{\Bigl(
\sum_k a^1_{ik} a^{1 *}_{ik}\,,
\sum_k a^2_{ik} a^{2 *}_{ik}\Bigr)}_{L^p\tiny(\N)}^{\frac12}\\
&=\,\biggl(\Bignorm{
\sum_k a^1_{ik} a^{1 *}_{ik}}_{L^p\tiny(\N_1)}^p\, +\,
\Bignorm{
\sum_k a^2_{ik} a^{2 *}_{ik}}_{L^p\tiny(\N_2)}^p\biggr)^{\frac12}\\
&=\,\Bigl(
\norm{(a^1_{ik})_{i,k}}_{L^{2p}(\tiny{\N}_1;\{\ell^2_{nm}\}_r)}^{2p}
\,+\, \norm{(a^2_{ik})_{i,k}}_{L^{2p}(\tiny{\N}_2;\{\ell^2_{nm}\}_r)}^{2p}
\Bigr)^{\frac12}\\
&\leq \,\Bigl(\bigl(\norm{[x_{ij}^1]}_{L^p(\tiny{\N_1}
;S^1_n)}+\varepsilon\bigr)^p
+\bigl(\norm{[x_{ij}^2]}_{L^p(\tiny{\N_2}
;S^1_n)}+\varepsilon\bigr)^p\Bigr)^{\frac12}.
\end{align*}
Likewise,
$$
\norm{(b^1_{kj}, b^2_{kj})_{i,k}}_{L^{2p}(\tiny{\N};\{\ell^2_{mn}\}_c)}
\leq \,\Bigl(\bigl(\norm{[x_{ij}^1]}_{L^p(\tiny{\N_1}
;S^1_n)}+\varepsilon\bigr)^p
+\bigl(\norm{[x_{ij}^2]}_{L^p(\tiny{\N_2}
;S^1_n)}+\varepsilon\bigr)^p\Bigr)^{\frac12}.
$$
Since $(x_{ij}^1,x_{ij}^2) = \sum_k (a^1_{ik}, a^2_{ik})
(b^1_{kj}, b^2_{kj})$ for all $1\leq i,j\leq n$ and 
$\varepsilon>0$ is arbitrary, the above two estimates 
imply the inequality $\leq$ in (\ref{Sum-S1-bis}).
The proof of the reverse inequality is similar.
\end{proof}

The next result may be known to operator space specialists. We include a 
proof for the sake of completeness.

\begin{lemma}\label{5.5}
Let $1\leq p\leq\infty$, let $n\geq 2$ and 
let $t\colon S^p_n\to S^p_n$ denote the transposition operator. We have
\begin{itemize}
\item[(i)] $\|t:S^p_n\to S^p_n\|_{cb}=
\norm{I_{S^p_n}\otimes t\colon S^p_n[S^p_n]
\to S^p_n[S^p_n]}=\,n^{2 \vert\frac12 -\frac1p \vert}$;
\item[(ii)] $\|t:S^p_n\to S^p_n\|_{reg}=
\norm{t\otimes I_{S^1_n}\colon S^p_n[S^1_n]
\to S^p_n[S^1_n]}=\,n$.
\end{itemize}
\end{lemma}

\begin{proof}
We will use the the Haagerup tensor product 
$\mathop{\otimes}\limits^h$, the row and column
operator spaces $R_n$ and $C_n$, the interpolation spaces $R_n(\theta)
=(C_n,R_n)_\theta$ for $\theta\in[0,1]$, 
introduced in \cite{P6}, and the construction of 
operator space valued $S^p$-spaces from \cite[Chapter 1]{P5}. 
We will also use the crucial fact that the Haagerup tensor product
commutes with interpolation (see \cite[Theorem 2.3]{P6} for a precise statement).
We refer to the above references and to \cite{BLM,P3} for some background.

Let $(e_1,\ldots,e_n)$ be the standard basis of $\ell^2_n$. 
It follows from \cite[Theorem 1.1]{P5} that for any 
operator space $E$, the mapping $E_{ij}\otimes 
x\mapsto e_i\otimes x\otimes e_j$, $1\leq i,j\leq n$ and $x\in E$, 
uniquely extends to a completely isometric isomorphism 
\begin{align}\label{h2}
S_n^p[E]\simeq R_n\bigl(\tfrac{1}{p}\bigr)
\mathop{\otimes}\limits^h E\mathop{\otimes}\limits^h 
R_n\bigl(1-\tfrac{1}{p}\bigr).
\end{align}

\smallskip\noindent
$(i):$ First we note that $\|t\colon M_n\to M_n\|_{cb}=n$, see e.g. \cite[Proposition 2.2.7]{ER}. Since we have
$\|t\colon S^2_n\to S^2_n\|_{cb}=1$, we obtain 
by interpolation that
$$
\|t\colon S^p_n\to S^p_n\|_{cb}\leq n^{2 \vert\frac12 -\frac1p \vert}.
$$ 
We now turn to lower estimates.
Consider the matrix $[E_{ij}]$ in $S_n^p[S_n^p]$ 
and note that $I_{S^p_n}\otimes
t$ maps $[E_{ij}]$ to $[E_{ji}]$ and 
$[E_{ji}]$ to $[E_{ij}]$. 
Applying (\ref{h2}) with $E=S^p_n$ equipped with its canonical operator
space structure, we have isometric identifications
\begin{align*}
S^p_n[S^p_n] & \simeq R_n\bigl(\tfrac{1}{p}\bigr)
\mathop{\otimes}\limits^h R_n\bigl(\tfrac{1}{p}\bigr)
\mathop{\otimes}\limits^h 
R_n\bigl(1-\tfrac{1}{p}\bigr)
\mathop{\otimes}\limits^h 
R_n\bigl(1-\tfrac{1}{p}\bigr)\\
&\simeq R_{n^2}\bigl(\tfrac{1}{p}\bigr)
\mathop{\otimes}\limits^h 
R_{n^2}\bigl(1-\tfrac{1}{p}\bigr) \\
&\simeq S^p_{n^2}.
\end{align*}
In the first of these identifications,
$[E_{ij}]$ corresponds to
$\sum_{i,j} e_i\otimes e_i\otimes e_j\otimes e_j$, which
may be written as $\bigl(\sum_i e_i\otimes e_i\bigr)
\otimes\bigl(\sum_j e_j\otimes e_j\bigr)$.
Since the $e_i\otimes e_i$ are pairwise orthogonal 
in $\ell_{n^2}^2$, we deduce that 
$$
\|[E_{ij}]\|_{S_n^p[S_n^p]}\,=\,
\Bignorm{\sum_{i=1}^n e_i\otimes 
e_i}_{R_{n^2}(\frac{1}{p})}
\,\Bignorm{\sum_{j=1}^n e_j\otimes 
e_j}_{R_{n^2}(1-\frac{1}{p})} = n^\frac12 n^\frac12 = n.
$$
Similarly, $[E_{ji}]$ corresponds to
$\sum_{i,j} e_i\otimes e_j\otimes e_i\otimes e_j$.
Further $\{e_i\otimes e_j\,:\, 1\leq i,j\leq n\}$ is 
an orthonormal basis of $\ell^2_{n^2}$. Hence  
through the identification of $S^p_n[S^p_n]$ with $S^p_{n^2}$, 
$[E_{ji}]$ corresponds to the identity map on $\ell^2_{n^2}$.
Its $S^p$-norm is equal to $n^{\frac{2}{p}}$, hence
$$
\|[E_{ji}]\|_{S_n^p[S_n^p]}=n^{\frac{2}{p}}.
$$
These computations show that 
$\|I_{S^p_n}\otimes t
\colon S_n^p[S_n^p]\to S_n^p[S_n^p]\| \geq n^{2\lvert 1/2-1/p\rvert}$.
Since the cb-norm of $t$ is greater than or equal to
$\|I_{S^p_n}\otimes t
\colon S_n^p[S_n^p]\to S_n^p[S_n^p]\|$, this proves the double equality in (i).

\smallskip
\noindent
$(ii):$	Note that
$$
\|t\colon M_n\to M_n\|_{reg}= \|t\colon M_n\to M_n\|_{cb} =n 
$$
and that $\|t:S^1_n\to S^1_n\|_{reg}= \|t\colon M_n\to M_n\|_{reg}$ by duality.
Hence by interpolation,
$$
\|t\colon S^p_n\to S^p_n\|_{reg}\leq n.
$$ 

We now turn to lower estimates.
We have $S^1_n\simeq R_n\mathop{\otimes}\limits^h C_n$ completely 
isometrically hence applying (\ref{h2}) with $E=S^1_n$,  we have
an isometric identification
$$
S^p_n[S^1_n]
\simeq R_n\bigl(\tfrac{1}{p}\bigr)
\mathop{\otimes}\limits^h R_n
\mathop{\otimes}\limits^h 
C_n
\mathop{\otimes}\limits^h 
R_n\bigl(1-\tfrac{1}{p}\bigr).
$$
According to e.g. \cite[Proposition 1.5.14 (6) $\&$ (8)]{ER}, we have
$$
R_n\bigl(\tfrac{1}{p}\bigr)
\mathop{\otimes}\limits^h R_n
\simeq \bigl(C_n
\mathop{\otimes}\limits^h R_n, R_n\mathop{\otimes}\limits^h R_n\bigr)_{\frac1p}
\simeq (M_n,S^2_n)_{\frac1p} = S^{2p}_n.
$$
Likewise,
$$
C_n
\mathop{\otimes}\limits^h 
R_n\bigl(1-\tfrac{1}{p}\bigr)\simeq 
\bigl(C_n\mathop{\otimes}\limits^h C_n, C_n
\mathop{\otimes}\limits^h R_n\bigr)_{1-\frac1p} 
\simeq (S^2_n,M_n)_{1-\frac1p}  = S^{2p}_n.
$$
Hence arguing as in the proof of (i), we have
\begin{align*}
\|[E_{ij}]\|_{S_n^p[S_n^1]}\,&=\,
\Bignorm{\sum_{i=1}^n e_i\otimes e_i}_{R_n(\frac1p)\mathop{\otimes}\limits^h R_n}
\Bignorm{\sum_{j=1}^n e_j\otimes e_j}_{C_n\mathop{\otimes}\limits^h R_n(1-\frac{1}{p})}\\
&=\,\bignorm{I_n\colon
\ell_n^2\to \ell_n^2}_{S_n^{2p}}^2 
=n^{\frac{1}{p}}.
\end{align*}
Next using as in (i) the correspondance between $[E_{ji}]$ and 
$\sum_{i,j} e_i\otimes e_j\otimes e_i\otimes e_j$, 
as well as the functorial property 
of the Haagerup tensor product (see e.g. \cite[1.5.5]{BLM}), we have
\begin{align*}
\Bignorm{\sum_{i,j=1}^n 
e_i\otimes e_j\otimes & e_i\otimes e_j}_{R_n
\mathop{\otimes}\limits^h R_n \mathop{\otimes}\limits^h C_n
\mathop{\otimes}\limits^h C_n}\\ & \leq
\bignorm{I_n\colon R_n\bigl(\tfrac{1}{p}\bigr)\to R_n}_{cb}
\bignorm{I_n\colon C_n\to R_n\bigl(1-\tfrac{1}{p}\bigr)}_{cb}\,
\|[E_{ji}]\|_{S_n^p[S_n^1]}.
\end{align*}
Using the facts that $CB(C_n,R_n)\simeq S^2_n$ and
$CB(C_n,C_n)\simeq M_n$ (see e.g. 
\cite[Section 4]{ER0}),
we both have $\bignorm{I_n\colon C_n\to R_n}_{cb} = n^\frac12$
and $\bignorm{I_n\colon C_n\to C_n}_{cb} = 1$. Hence 
$$
\bignorm{I_n\colon R_n\bigl(\tfrac{1}{p}\bigr)\to R_n}_{cb}
\leq n^{\frac12(1-\frac1p)},
$$
by interpolation. Likewise,
$$
\bignorm{I_n\colon C_n\to R_n\bigl(1-\tfrac{1}{p}\bigr)}_{cb}
\leq n^{\frac12(1-\frac1p)}
$$
Further $R_n
\mathop{\otimes}\limits^h R_n \mathop{\otimes}\limits^h C_n
\mathop{\otimes}\limits^h C_n\simeq R_{n^2}
\mathop{\otimes}\limits^h C_{n^2} \simeq S^1_{n^2}$ hence
$$
\Bignorm{\sum_{i,j=1}^n 
e_i\otimes e_j\otimes e_i\otimes e_j}_{R_n
\mathop{\otimes}\limits^h R_n \mathop{\otimes}\limits^h C_n
\mathop{\otimes}\limits^h C_n} \,
=\,\bignorm{I_{n^2}\colon \ell^2_{n^2}\to \ell^2_{n^2}}_1 = n^2.
$$
These estimate yield
$$
\|[E_{ji}]\|_{S_n^p[S_n^1]}\geq n^{1+\frac1p}.
$$
Hence we obtain that
$$
\norm{t\otimes I_{S^1_n}\colon S^p_n[S^1_n]
\to S^p_n[S^1_n]}\geq \frac{n^{1+\frac1p}}{n^\frac{1}{p}}=n.
$$
Since $\norm{t\colon S^p_n\to S^p_n}_{reg}\geq \norm{t\otimes I_{S^1_n}\colon S^p_n[S^1_n]
\to S^p_n[S^1_n]}$, (ii) follows at once.
\end{proof}

\begin{theorem}\label{jan5}
Let $T\colon L^p(\mathcal{M})\to L^p(\mathcal{N})$ be an isometry.
The following statements are equivalent.
\begin{itemize}
\item[(i)] $T$ admits a direct Yeadon type factorization.
\item[(ii)] $T$ is $S^1$-contractive.
\end{itemize}
\end{theorem}

\begin{proof}
The implication ``$(i)\Rightarrow (ii)$" follows from Proposition \ref{4.2}
so we only need to prove ``$(ii)\Rightarrow (i)$". 

We first show this implication in the case when $\M=M_n$, with $n\geq 2$. 
Let $T\colon L^p(M_n)\to L^p(\N)$ be an isometry and assume
that $T$ is $S^1$-contractive. By Remark \ref{l1}, \cite[Theorem 4.2]{LMZ} and
Yeadon's Theorem,
$T$ admits a Yeadon type factorisation.
Let $(w,B,J)$ be its Yeadon triple. 
Changing $T$ into $w^*T(\,\cdotp)$, see Remark \ref{Tensorization} (a),
we can assume that $w=J(1)$.
Consider a decomposition 
$J=\begin{pmatrix}\pi&0\\0&\sigma
\end{pmatrix}$ as in (\ref{cd}). We aim at showing  that $\sigma=0$.

Let us apply Remark \ref{cdec} to $T$. In the sequel 
we use the elements 
$\N_1,\N_2$, $B_1,B_2$ and 
$$
T_1\colon L^p(M_n)\longrightarrow L^p(\N_1),\quad
T_2\colon L^p(M_n)\longrightarrow L^p(\N_2)
$$ 
from this remark. By construction we have $T_1(x)=B_1\pi(x)$
and $T_2(x)=B_2\sigma(x)$ for any $x\in L^p(M_n)$.

Applying Lemma \ref{last} to the unital
$*$-homomorphism $\pi\colon M_n\to \N_1$,
we obtain a projection $\epsilon_1$ in $\N_1$ and a 
bijective $*$-homomorphism 
$\rho_{\pi}:\N_1\to M_n\overline{\otimes}
(\epsilon_1\N_1\epsilon_1)$ such that
$$ 
\left(\rho_{\pi}\circ\pi\right)(x)=x\otimes\epsilon_1,\qquad
x\in M_n,
$$
and $\rho_{\pi}$ is trace preserving. By Lemma \ref{TP-maps}, 
$\rho_{\pi}$
induces an isometry (still denoted by)
$$
\rho_{\pi}\colon L^p(\N_1)\longrightarrow L^p(M_n\overline{\otimes}
(\epsilon_1\N_1\epsilon_1))\,\simeq\, 
S^p_n\otimes L^p(\epsilon_1\N_1\epsilon_1).
$$

We have $B_1=T_1(I_n)$, hence $B_1\in L^p(\N_1)$. 
Further for any $x\in L^p(M_n)$, we have
\begin{align*}
\left(\rho_{\pi}\circ T_1\right)(x)&=\rho_{\pi}\left(B_1\pi(x)\right)\\
&=\rho_{\pi}(B_1)\rho_{\pi}(\pi(x))\\
&=\rho_{\pi}(B_1)(x\otimes \epsilon_1).
\end{align*}
Since $B_1\pi(x)=\pi(x)B_1$, a similar computation shows that we also have
$\left(\rho_{\pi}\circ T_1\right)(x)=(x\otimes \epsilon_1)\rho_{\pi}(B_1)$.
This shows that $\rho_{\pi}(B_1)$ commutes with $x\otimes\epsilon_1$ 
for any $x\in L^p(M_n)$. Consequently
there exists $b_1$ in $L^p(\epsilon_1\N_1\epsilon_1)$ 
such that $\rho_{\pi}(B_1)=I_n\otimes b_1$. Then the above computation
shows that
\begin{equation}\label{b1}
\left(\rho_{\pi}\circ T_1\right)(x) = x\otimes b_1,\qquad x\in L^p(M_n).
\end{equation}

Recall that we let $t\colon M_n\to M_n$ denote the transposition
map. The mapping $\sigma\circ t\colon M_n\to \N_2$ is a 
unital
$*$-homomorphism. Hence arguing as above, we obtain 
a projection $\epsilon_2$ in $\N_2$, a trace
preserving bijective $*$-homomorphism
$\rho_\sigma\colon \N_2\to M_n\overline{\otimes}(\epsilon_2\N_2\epsilon_2)$,
inducing an isometry
$$
\rho_\sigma\colon L^p(\N_2)\longrightarrow 
L^p(M_n\overline{\otimes}(\epsilon_2\N_2\epsilon_2))\simeq S^p_n \otimes 
L^p(\epsilon_2\N_2\epsilon_2),
$$
and some $b_2$ in $L^p(\epsilon_2\N_2\epsilon_2)$, such that
\begin{equation}\label{b2}
\left(\rho_{\sigma}\circ T_2\right)(x)= t(x)\otimes b_2,
\qquad x\in L^p(M_n).
\end{equation}

Observe that
$\rho_{\pi}\colon L^p(\N_1)\to L^p(M_n\overline{\otimes}
(\epsilon_1\N_1\epsilon_1))$ and 
$\rho_{\sigma}\colon L^p(\N_2)\to L^p(M_n\overline{\otimes}
(\epsilon_2\N_2\epsilon_2))$ are completely positive. Hence by 
Theorem \ref{ls1}, they are $S^1$-contractive.

Let $m\geq 1$ and let $[x_{ij}]_{1\leq i,j\leq m}$ in $S^p_n[S^1_m]$.
Since 
$\rho_{\pi}$ is $S^1$-contractive, we have
\begin{align*}
\|[\rho_{\pi}\circ T_1(x_{ij})]\|_{L^p
(M_n\overline{\otimes}(\epsilon_1\tiny{\N_1}\epsilon_1)
;S^1_m)}\leq\|[T_1(x_{ij})]\|_{L^p(\tiny{\N_1};S^1_m)}.
\end{align*}
On the other hand, using (\ref{b1}) and
Lemma \ref{jan52}, we have 
$$
\|[\rho_{\pi}\circ T_1(x_{ij})]\|_{L^p(
M_n\overline{\otimes}(\epsilon_1\tiny{\N_1}\epsilon_1)
;S^1_m)}=\|[x_{ij}\otimes b_1]\|_{L^p(
M_n\overline{\otimes}(\epsilon_1\tiny{\N_1}\epsilon_1)
;S^1_m)} = \|[x_{ij}]\|_{S^p_n[S^1_m]}\|b_1\|_p.
$$
Hence we obtain that
$$
\|b_1\|_p\|[x_{ij}]\|_{S^p_n[S^1_m]}
\leq\|[T_1(x_{ij})]\|_{L^p(\tiny{\N_1};S^1_m)}.
$$
Similarly, using (\ref{b2}), we have
$$
\|b_2\|_p\|[t(x_{ij})]\|_{S_n^p[S^1_m]}
\leq\|[T_2(x_{ij})]\|_{L^p(\tiny{\N_2};S^1_m)}.
$$
Taking the $p$-th powers and summing the above inequalities, we obtain that
\begin{align*}
\|b_1\|_p^p\|[x_{ij}]\|_{S^p_n[S^1_m]}^p+&\|b_2\|_p^p \|
[t(x_{ij})]\|_{S^p_n[S^1_m]}^p \\
& \leq\, \|
[T_1(x_{ij})]\|_{L^p(\tiny{\N_1};S^1_m)}^p+
\|[T_2(x_{ij})]\|_{L^p(\tiny{\N_2};S^1_m)}^p.
\end{align*}
According to Lemma \ref{Sum-S1}, the right-hand side
in the above inequality coincides
with $\|[T(x_{ij})]\|_{L^p(\tiny{\N};S^1_m)}^p$.
Since $T$ is assumed $S^1$-contractive, we infer that
\begin{equation}\label{p-sum}
\|b_1\|_p^p\|[x_{ij}]\|_{S^p_n[S^1_m]}^p+\|b_2\|_p^p \|
[t(x_{ij})]\|_{S^p_n[S^1_m]}^p \leq 
\|[x_{ij}]\|_{S^p_n[S^1_m]}^p.
\end{equation}

Using (\ref{b1}) and (\ref{b2}) again, we note that
for any $x\in S^p_n$, 
\begin{align*}
\norm{T(x)}_p^p 
& = \norm{T_1(x)}_p^p + \norm{T_2(x)}_p^p\\
& = \norm{x\otimes b_1}_p^p + \norm{t(x)\otimes b_2}_p^p,
\end{align*}
and hence
\begin{equation}\label{sep=iso}
\norm{T(x)} = \norm{x}_p^p\bigl(\norm{b_1}_p^p + \norm{b_2}_p^p\bigr).
\end{equation}
Since $T$ is an isometry, this implies that 
$$
\|b_1\|_p^p+\|b_2\|_p^p=1.
$$
Replacing $\|b_1\|_p^p$ by $(1- \|b_2\|_p^p)$ in (\ref{p-sum}),
we obtain that
$$
\|b_2\|_p\|[t(x_{ij})]\|_{S^p_n[S^1_m]}
\leq\|b_2\|_p\|[x_{ij}]\|_{S_n^p[S^1_m]}	
$$
for any $m\geq 1$ and any $[x_{ij}]_{1\leq i,j\leq m}$ in $S^p_n[S^1_m]$.
By Lemma \ref{5.5} (ii),
the above inequality holds only if $b_2=0$. In this case, 
we have $\sigma=0$, and hence $J$ is a $*$-homomorphism.

We now consider the general case. We let 
$T\colon L^p(\mathcal{M})\to L^p(\mathcal{N})$ be an isometry
and assume that  $T$ is $S^1$-contractive. As in the first part of the proof,
this implies that $T$ has a Yeadon type factorisation. Let
$J\colon\M\to\N$ be the Jordan homomorphism in the Yeadon triple 
of $T$ and let
$J=\begin{pmatrix}\pi&0\\0&\sigma
\end{pmatrix}$ be a decomposition of $J$ as in (\ref{cd}).
Let $\mathcal{M}_1 =\text{Ker}(\sigma)$. 
Since $\sigma$ is $w^*$-continuous, $\M_1$ is a
$w^*$-closed ideal of $\M$. Hence we have a direct sum decomposition
$$
\M=\mathcal{M}_1\mathop{\oplus}\limits^\infty\mathcal{M}_2.
$$
Moreover $\sigma_{|{\mathcal{M}_2}}$ is one-to-one.
To prove that $J$ is a $*$-homomorphism, 
it suffices to show that $\mathcal{M}_2$ is abelian.

If not, then
by Lemma \ref{Contain-M2}, there 
exists a non zero $*$-homomorphism
$\gamma\colon M_2\to \M_2$ taking values in
$\M_2\cap L^1(\M_2)$.  Let $\tau'=\tau_{\tiny\M}\circ\gamma \colon M_2\to \Cdb$.
Then $\tau'$ is a non zero
trace on $M_2$ hence there exists $\delta>0$
such that $\tau'=\delta tr_2$. This readily implies that  
$$
\delta^{-\frac1p}\gamma\colon L^p(M_2)\longrightarrow L^p(\M_2) 
$$
is an isometry. Further $\delta^{-\frac1p}\gamma$ is completely positive.
Hence by Theorem \ref{ls1}, $\delta^{-\frac1p}\gamma$ is $S^1$-contractive. 
By composition, we obtain that 
$\delta^{-\frac1p} T\circ \gamma$ is an $S^1$-contractive isometry
from $L^p(M_2)$ into $L^p(\N)$. According to the first part of this proof, 
$\delta^{-\frac1p}T\circ \gamma$ has therefore a direct Yeadon type factorization.
We observe that the Jordan homomorphism of its Yeadon triple
is equal to $J\circ \gamma$. The latter is therefore multiplicative,
hence 
$\sigma\circ \gamma$ is multiplicative. Since $\sigma\circ \gamma$
also is anti-multiplicative, we actually have
$$
\sigma\circ \gamma (ab) = [\sigma\circ \gamma (b)][
\sigma\circ \gamma(a)]=\sigma\circ \gamma(ba)
$$
for any $a,b\in M_2$. However $\sigma\circ \gamma$ is one-to-one,
hence the above property implies that
$ab=ba$ for any $a,b\in M_2$, a contradiction.
Hence $\mathcal{M}_2$ is abelian as expected, which concludes the proof.
\end{proof}

\begin{remark}
Let $1\leq p<\infty$ and let $\N$
be a semifinite 
von Neumann algebra. 
The argument in the first part of
the proof of Theorem \ref{jan5} shows that for any $n\geq 1$
and for any non zero separating map $T\colon S^p_n\to L^p(\N)$,
the operator $\norm{T}^{-1} T$ is an isometry. Indeed this follows
from (\ref{sep=iso}).

Likewise for any Hilbert space $\H$ and 
for any non zero separating map $T\colon S^p(\H)\to L^p(\N)$,
the operator $\norm{T}^{-1} T$ is an isometry.
\end{remark}

\begin{theorem}\label{5.6}
Let $1\leq p\neq 2<\infty$ and let 
$T\colon L^p(\mathcal{M})\to L^p(\mathcal{N})$
be an isometry. The following statements are equivalent.
\begin{itemize}
\item[(i)] $T$ admits a direct Yeadon type factorization.
\item[(ii)] $T$ is completely contractive.
\end{itemize} 
\end{theorem}

\begin{proof}
The implication ``$(i)\Rightarrow (ii)$" follows from Proposition \ref{4.4}
so we only need to prove ``$(ii)\Rightarrow (i)$". It turns out that
the proof of the similar implication in Theorem \ref{jan5} applies 
for this case, up to a few changes that we now explain.

Assume first that $\M=M_n$, with $n\geq 2$, and
consider $T_1, T_2, \rho_\pi,\rho_\sigma, b_1,b_2$ given by
the proof of Theorem \ref{jan5}. By Lemma \ref{TP-maps},
$\rho_{\pi}\colon L^p(\N_1)\to L^p(M_n\overline{\otimes}
(\epsilon_1\N_1\epsilon_1))$ and 
$\rho_{\sigma}\colon L^p(\N_2)\to L^p(M_n\overline{\otimes}
(\epsilon_2\N_2\epsilon_2))$ are complete isometries.
Further for any $m\geq 1$ and any $[x_{ij}]_{1\leq i,j\leq m}$
in $S^p_m[S^p_n]$, we have
$$
\norm{[x_{ij}\otimes b_1}_{L^p(M_m\overline{\otimes}
M_n\overline{\otimes}(\epsilon_1\tiny{N_1}\epsilon_1))}
\,=\, \norm{[x_{ij}}_{S^p_m[S^p_n]}\norm{b_1}_p,
$$
by (\ref{1-tensor}). Hence 
$$
\|b_1\|_p\|[x_{ij}]\|_{S^p_m[S^p_n]}
\leq\|[T_1(x_{ij})]\|_{L^p(M_m\overline{\otimes}\tiny{\N_1})}.
$$
Similarly
$$
\|b_2\|_p\|[x_{ij}]\|_{S^p_m[S^p_n]}
\leq\|[T_2(x_{ij})]\|_{L^p(M_m\overline{\otimes}\tiny{\N_2})}.
$$
Moreover by (\ref{DirectSum}),
$$
\norm{[T(x_{ij})]}^p_{L^p(M_m\overline{\otimes}\tiny{\N})}
\,=\, 
\norm{[T_1(x_{ij})]}^p_{L^p(M_m\overline{\otimes}\tiny{\N_1})}\,+\,
\norm{[T_2(x_{ij}]}^p_{L^p(M_m\overline{\otimes}\tiny{\N_2})}.
$$
Then using Lemma \ref{5.5} (i), 
the argument in the proof Theorem \ref{jan5} shows that 
$b_2=0$ and hence that $T$ has a direct Yeadon type factorization.

In the general case, the proof of Theorem \ref{jan5} applies almost verbatim,
using the simple fact that $\delta^{-\frac1p}\gamma$ is a complete isometry.
\end{proof}

\begin{remark}\label{ex0}
Let
$n\geq 2$ and consider $T\colon S^p_n\mathop{\oplus}\limits^p S^p_n \to 
S^p_n\mathop{\oplus}\limits^p S^p_n$  defined by 
$$
T(x,y) =\bigl(x,n^{-\frac1p}t(x)\bigr),\qquad
x,y\in S^p_n.
$$
Then $T$ is a separating map and by Lemma \ref{5.5}, we have 
$\norm{T}=\norm{T}_{S^1}=\norm{T}_{cb}$. However 
$T$ does not have a direct Yeadon type factorization.
This shows that Theorems \ref{jan5} and \ref{5.6} cannot hold true if we remove the
isometric assumption on $T$.
\end{remark}

\begin{remark}\label{ex0} 
Let $T\colon L^p(\mathcal{M})\to L^p(\mathcal{N})$ be an isometry.
The proof of Theorem \ref{jan5} actually shows that 
$T$ admits a direct Yeadon type factorization if and only if
$T$ is $S^1_2$-contractive, that is, 
$$
\bignorm{T\otimes I_{S^1_2}\colon L^p(\M; S^1_2)\longrightarrow L^p(\N; S^1_2)}\leq 1.
$$
Likewise if $p\not=2$, the proof of Theorem \ref{5.6} 
shows that 
$T$ admits a direct Yeadon type factorization if and only if
$T$ is $2$-contractive.
\end{remark}

Note that  
Theorem \ref{5.6} and the above remark extend \cite[Proposition 3.2]{JRS}.
Theorem \ref{jan5} can be regarded as a variant of the latter. Its main
feature is that it also applies to $p=2$. We emphasize this in the next statements.

\begin{corollary}\label{JRS-L2}
An isometry $T\colon L^2(\M)\to L^2(\N)$ 
admits a direct Yeadon type factorization if and only if
it is $S^1$-contractive.
\end{corollary}

\begin{corollary}\label{cpi}
Any completely positive isometry $T\colon L^2(\M)\to L^2(\N)$ 
admits a direct Yeadon type factorization.
\end{corollary}
\begin{proof}
This follows from Theorem \ref{ls1} and Theorem \ref{jan5}.
\end{proof}

\begin{remark}
Assume here that $\M,\N$ are semifinite and hyperfinite von Neumann algebras.
In the case when $p\not=2$, Theorem \ref{jan5} follows from Theorem \ref{5.6},
by Proposition \ref{3.8} and \cite[Proposition 2.2]{P2}. Moreover the
$L^2$-case of Theorem \ref{jan5}, and hence
Corollaries \ref{JRS-L2} and \ref{cpi}, have a much simpler proof.
Indeed under the hyperfinite assumption, suppose that $T\colon L^2(\mathcal{M})\to L^2(\mathcal{N})$ 
is an $S^1$-contractive isometry. By Proposition \ref{3.8},
$T$ is completely regular with $\norm{T}_{reg}\leq 1$. Applying (\ref{Reg-E})
with the specific operator space $E=S^2_2[{\rm Max}(\ell^1_2)]$ we obtain that 
\begin{equation}\label{Simple}
\bignorm{T\otimes I_{S^2_2}\otimes I_{\ell^1_2}\colon
L^2(\M)\bigl[S^2_2[{\rm Max}(\ell^1_2)]\bigr]\longrightarrow 
L^2(\N)\bigl[S^2_2[{\rm Max}(\ell^1_2)]\bigr]}\leq 1.
\end{equation}
According to \cite[Theorem 1.9]{P5}, we have a Fubini type isometric identification
between $L^2(\M)\bigl[S^2_2[{\rm Max}(\ell^1_2)]\bigr]$
and $L^2(M_2\overline{\otimes}\M)
[{\rm Max}(\ell^1_2)]$. Combining with \cite[(7)]{LMZ}, we then have 
$$
L^2(\M)\bigl[S^2_2[{\rm Max}(\ell^1_2)]\bigr]\simeq L^2(M_2\overline{\otimes}\M;\ell^1_2).
$$
We have a similar result for $\N$. Consequently 
(\ref{Simple}) implies that 
$$
I_{S^2_2}\otimes T\colon L^2(M_2\overline{\otimes}\M)\longrightarrow
L^2(M_2\overline{\otimes}\N)
$$
is $\ell^1_2$-contractive. Further $L^2(M_2\overline{\otimes}\M)$ (resp. 
$L^2(M_2\overline{\otimes}\N)$) coincides with the Hilbertian tensor product
of $S^2_2$ and $L^2(\M)$ (resp. $L^2(\N)$). Hence 
$I_{S^p_2}\otimes T$ is an isometry.
It therefore follows from \cite[Theorem 4.2]{LMZ} that 
$I_{S^2_2}\otimes T$ admits a Yeadon type factorization. 
By \cite[Theorem 3.6]{HRW}, this implies that $T$ admits a direct Yeadon type factorization.
\end{remark}

\vskip 0.5cm
\noindent
{\bf Acknowledgement.} 
The work leading to this paper started whilst the second author was visiting 
``Laboratoire de Math\'ematiques de Besan\c con" (LmB). She greatly acknowledges 
LmB for hospitality and excellent working condition.  The first named author is 
supported by the French ``Investissement d'Avenir" program, project
ISITE-BFC (contract ANR-15-IDEX-03).


\end{document}